\documentclass[11pt,a4paper]{article}

\usepackage{mathrsfs} 

\usepackage{graphicx}     
\usepackage[a4paper, hmargin={3cm,3cm},vmargin={3.3cm,3.3cm}]{geometry}
\usepackage{amsmath, amssymb, amsfonts, amsbsy, latexsym, color,dsfont}
\usepackage{amsthm}
\usepackage{mathtools}

\usepackage[T1]{fontenc}
\usepackage[latin1]{inputenc}
\usepackage[english]{babel}

\usepackage{cite}
\usepackage{paralist}
\usepackage{url}
\usepackage[bf,compact,pagestyles,small]{titlesec}
\titlespacing*{\section}{0pt}{14pt}{4pt}
\titlespacing*{\subsection}{0pt}{8pt}{3pt}
\usepackage{mathdots}
\usepackage{tikz}
\usetikzlibrary{calc}

\usepackage{fancyhdr,lastpage}
\def\maketimestamp{\count255=\time
\divide\count255 by 60\relax
\edef\thetime{\the\count255:}%
\multiply\count255 by-60\relax
\advance\count255 by\time
\edef\thetime{\thetime\ifnum\count255<10 0\fi\the\count255}
\edef\thedate{\number\day-\ifcase\month\or Jan\or Feb\or Mar\or
             Apr\or May\or Jun\or Jul\or Aug\or Sep\or Oct\or
             Nov\or Dec\fi-\number\year}
\def\timstamp{\hbox to\hsize{\tt\hfil\thedate\hfil\thetime\hfil}}}
\maketimestamp
\fancypagestyle{plain}{%
\fancyhf{} 
\fancyfoot[C]{\small \thepage{} of \pageref{LastPage}}}

\numberwithin{equation}{section}  
\allowdisplaybreaks[4]

\newtheorem{theorem}{Theorem}[section]
\newtheorem{lemma}[theorem]{Lemma}
\newtheorem{proposition}[theorem]{Proposition}
\newtheorem{corollary}[theorem]{Corollary}
\theoremstyle{definition}
\newtheorem{definition}[theorem]{Definition} 
\newtheorem{example}{Example}
\theoremstyle{remark}
\newtheorem{remark}{Remark}

\DeclareMathOperator{\Span}{span} %
\DeclareMathOperator{\supp}{supp} %
\DeclareMathOperator{\volumeet}{d\!} %
\DeclareMathOperator{\volumeto}{\tilde{d}\!} %

\DeclareMathOperator{\exponential}{e}

\newcommand{\sepgabor}[3][g]{\ensuremath{\set{E_{\gamma}T_{\lambda}
      {#1}}_{\lambda\in{#2},\gamma\in {#3}}}} 
\newcommand{\gabor}[2][\Delta]{\{\pi(\nu)#2\}_{\nu \in #1}}
\newcommand{\gaborG}[2][\Delta]{\mathscr{G}(#2,#1)} 

\newcommand{\myexp}[1]{\exponential^{#1}}

\newcommand{\id}[1][]{I_{#1}} 

\newcommand{\vol}[1]{\volumeet{(#1)}}
\newcommand{\volto}[1]{\volumeto{(#1)}}

\newcommand*{\numbersys}[1]{\ensuremath{\mathbb{#1}}}
\newcommand*{\C}{\numbersys{C}}
\newcommand*{\R}{\numbersys{R}}

\newcommand*{\Z}{\numbersys{Z}}

\newcommand*{\T}{\numbersys{T}}

\newcommand*{\cH}{\mathcal{H}}

\newcommand*{\cF}{\mathcal{F}}

\newcommand*{\cK}{\mathcal{K}}

\newcommand{\itvcc}[2]{\ensuremath{\left[{#1},{#2}\right]}} %
\newcommand{\itvco}[2]{\ensuremath{\left[{#1},{#2}\right)}} %
\newcommand{\abs}[1]{\ensuremath{\left\lvert#1\right\rvert}}

\newcommand{\norm}[2][]{\ensuremath{\left\lVert#2\right\rVert_{#1}}}

\newcommand{\innerprod}[3][]{\ensuremath{\left\langle #2,#3\right\rangle_{\! #1}}}

\newcommand{\innerprods}[2]{\ensuremath{\langle #1,#2\rangle}}

\newcommand{\set}[1]{\ensuremath{\left\lbrace{#1}\right\rbrace}}

\newcommand{\setprop}[2]{\ensuremath{\left\lbrace{#1} : {#2}\right\rbrace}}

\newcommand{\floor}[1]{\left\lfloor #1 \right\rfloor}

\newcommand{\lat}[1]{\ensuremath {#1}} 

\newcommand{\LG}{\ensuremath\lat{\Gamma}}
\newcommand{\LL}{\ensuremath\lat{\Lambda}}

\usepackage{xspace}
\newcommand{\ie}{i.e.,\xspace} 
\newcommand{\eg}{e.g.,\xspace} 

\newcommand{\cfm}{\mu_M} 
\newlength{\dhatheight}
\newcommand{\doublehat}[1]{%
	\settoheight{\dhatheight}{\ensuremath{\widehat{#1}}}%
	\addtolength{\dhatheight}{-0.35ex}%
    	\widehat{\vphantom{\rule{1pt}{\dhatheight}}%
    	\smash{\hspace{-2pt} \widehat{#1}}}}

\newcommand{\ghat}{\widehat{G}}
\newcommand{\ghhat}{\doublehat{G}}
\newcommand*\oline[1]{%
  \vbox{%
    \hrule height 0.5pt
    \kern0.25ex
    \hbox{%
      \kern-0.1em
      \ifmmode#1\else\ensuremath{#1}\fi
      \kern-0.1em
    }
  }
}

{\nopagebreak\par\noindent\hbox to \hsize{\hrulefill}\par\noindent}

\usepackage{hyperref}
\hypersetup{
 pdfview={FitH},
 pdfstartview={FitH},
 pdfauthor = {Mads Sielemann Jakobsen, Jakob Lemvig},
 pdftitle = {Density and duality results for regular Gabor frames}, 
 pdfsubject = {harmonic analysis},
 pdfkeywords = {density, dual frames, duality principle, frame, Gabor
   system, Janssen representation, LCA group,
   Wexler-Raz biorthogonality relations},
 pdfcreator = {LaTeX with hyperref package},
 pdfproducer = PDFlatex}

\makeatletter
\def\blfootnote{\xdef\@thefnmark{}\@footnotetext} 
\def\subjclass{\xdef\@thefnmark{}\@footnotetext}
\long\def\symbolfootnote[#1]#2{\begingroup%
\def\thefootnote{\fnsymbol{footnote}}\footnote[#1]{#2}\endgroup} 
\if@titlepage
  \renewenvironment{abstract}{%
      \titlepage
      \null\vfil
      \@beginparpenalty\@lowpenalty
      \begin{center}%
        \bfseries \abstractname
        \@endparpenalty\@M
      \end{center}}%
     {\par\vfil\null\endtitlepage}
\else
  \renewenvironment{abstract}{%
      \if@twocolumn
        \section*{\abstractname}%
      \else
        \small
        \list{}{%
          \settowidth{\labelwidth}{\textbf{\abstractname:}}
          \setlength{\leftmargin}{50pt}
          \setlength{\rightmargin}{50pt}
          \setlength{\itemindent}{\labelwidth}
          \addtolength{\itemindent}{\labelsep}
        }
        \item[\textbf{\abstractname:}]

      \fi}
      {\if@twocolumn\else\endlist\fi}
\fi
\makeatother

\begin{document}

\title{Density and duality theorems for regular Gabor frames}

\date{April 16, 2015} 

 \author{Mads Sielemann Jakobsen\footnote{Technical University of Denmark, Department of Applied Mathematics and Computer Science, Matematiktorvet 303B, 2800 Kgs.\ Lyngby, Denmark, E-mail: \protect\url{msja@dtu.dk}}\phantom{$\ast$}, Jakob Lemvig\footnote{Technical University of Denmark, Department of Applied Mathematics and Computer Science, Matematiktorvet 303B, 2800 Kgs.\ Lyngby, Denmark, E-mail: \protect\url{jakle@dtu.dk}}}

 \blfootnote{2010 {\it Mathematics Subject Classification.} Primary
   42C15, Secondary: 42C40, 43A25, 43A32, 43A70.} \blfootnote{{\it Key words
     and phrases.} Density, dual frames, duality principle, frame, Gabor
   system, Janssen representation, LCA group,
   Wexler-Raz biorthogonality relations, Weyl-Heisenberg systems}

\maketitle 

\thispagestyle{plain}
\begin{abstract} 
We investigate Gabor frames on locally compact abelian groups with
time-frequency shifts along non-separable, closed subgroups of the phase
space. Density theorems in Gabor analysis state necessary conditions
for a Gabor system to be a frame or a Riesz basis, formulated only in
terms of the index subgroup.  In the classical results the subgroup is
assumed to be discrete.  We prove density theorems for general closed
subgroups of the phase space, where the necessary conditions are given
in terms of the ``size'' of the subgroup. From these density results
we are able to extend the classical Wexler-Raz biorthogonal relations
and the duality principle in Gabor analysis to Gabor systems with
time-frequency shifts along non-separable, closed subgroups
of the phase space. Even in the euclidean setting, our results are new.
\end{abstract}

\section{Introduction}
\label{sec:introduction}

Classical harmonic analysis on locally compact abelian (LCA) groups provides a natural framework for
many of the topics considered in modern time-frequency analysis. The setup is as follows. Let
$(G,\cdot)$ denote a second countable LCA group, and let $(\ghat,\cdot)$ denote its dual group,
consisting of all characters. One then defines the translation operator $T_{\lambda}$, $\lambda \in
G$, as
\[
 T_{\lambda}:L^2(G)\to L^2(G), \ (T_{\lambda}f)(x) = f(x\lambda^{-1}), \quad x\in G,
\]
and the modulation operator $E_{\gamma}$, $\gamma\in \ghat$, as
\[ 
E_{\gamma}:L^2(G)\to L^2(G), \ (E_{\gamma}f)(x) = \gamma(x) f(x), \quad x\in G.
\]
The central objects of this work are so-called regular Gabor systems in $L^2(G)$ with modulation and
translation along a closed subgroup $\Delta$ of $G \times \ghat$ generated by a window function $g
\in L^2(G)$; this is a collection of functions of the following form:
\[
\gaborG{g}:=\gabor{g}, \quad \text{where } \pi(\nu):=E_\gamma T_\lambda \text{ for }
\nu=(\lambda,\gamma) \in G \times \ghat.
\]
The tensor product $G\times \ghat$ is called the phase-space or the time-frequency plane, and
$\pi(\nu)g$ is a time-frequency shift of $g$.

We are interested in linear operators of the form
\[
 C_{g,\Delta}: L^2(G) \to L^2(\Delta), \quad C_{g,\Delta} f = \nu \mapsto
\innerprod{f}{\pi(\nu)g}
\] 
as well as their left-inverses (if they exist) and adjoints. The $C_{g,\Delta}$ transform is called an
\emph{analysis} operator, while its adjoint is called \emph{synthesis}. In the analysis process
$C_{g,\Delta} f$ of a function $f\in L^2(G)$, we obtain information of the time-frequency content in
the function $f$.

If the operator $C_{g,\Delta}$ is bounded below and above, we say that $\gaborG{g}$ is a Gabor frame
for $L^2(G)$. In case the two constants from these bounds can be
taken to be equal, we say that $\gaborG{g}$ is a
tight frame; if they can be
taken to be equal to one, $\gaborG{g}$ is said to be a Parseval frame.  One can show
that the property of being a frame allows for stable reconstruction of any $f \in L^2(G)$ from its
time-frequency information given by $C_{g,\Delta}f$. In particular, if $C_{g,\Delta}$ is bounded
from below and above, then there exists another function $h\in L^2(G)$ such that $C_{h,\Delta}$ is a
bounded operator and such that
\begin{equation*} 
\langle f_1,f_2 \rangle = \int_{\Delta} C_{g,\Delta}f_1(\nu) \, \overline{C_{h,\Delta}f_2(\nu)} \, d\nu  
\end{equation*}
for all $f_1,f_2\in L^2(G)$, where $d\nu$ denotes the Haar measure on $\Delta$. Two such Gabor
systems $\gaborG{g}$ and $\gaborG{h}$ are said to be \emph{dual} Gabor frames.
If $\gaborG{g}$ is a frame with $C_{g,\Delta}$ being surjective, we say that
$\gaborG{g}$ is a \emph{Riesz family}.

In case $\Delta=G\times \ghat$ the analysis operator $C_{g,\Delta}$ is the well-known
\emph{short-time Fourier transform}, usually written $\mathcal{V}_g$, which is an isometry
for any window function $g \in L^2(G)$ satisfying
$\norm{g}=1$. 
In the language of frame theory, $\gaborG[G\times\ghat]{g}$ is said to be a Parseval frame. However,
for other subgroups $\Delta \subsetneq G \times \ghat$ window functions $g\in L^2(G)$ leading to
isometric transforms $C_{g,\Delta}$, or more generally to Gabor frames $\gaborG{g}$ for $L^2(G)$,
might not exist.

The density theorems in Gabor analysis are such non-existence results formulated only as
\emph{necessary conditions on the subgroup $\Delta$} for a Gabor system to be a frame or a Riesz
basis. In particular, the subgroup $\Delta$ needs to possess a certain amount of density. The
classical density results are stated for \emph{uniform lattices} $\Delta$, i.e., discrete and
co-compact subgroups of $G \times \ghat$, where the density is measured by the volume of a
fundamental domain of $\Delta$. For $\Delta=P\Z^{2d}$, $P \in \mathrm{GL}_{2d}(\R)$, in $G \times
\ghat=\R^{2d}$, this volume is exactly $\abs{\det{P}}$. In this work
we introduce a generalization of
this density measure for non-lattices so, for closed subgroups
$\Delta$ of $G \times \ghat$, we set
\[ 
\vol{\Delta} :=  \mu_{(G \times \ghat)/\Delta}((G \times \ghat)/\Delta).
\]  
If $\Delta$ is a uniform lattice equipped with the counting measure, then $\vol{\Delta}$ is exactly
the measure of the fundamental domain. Note that $\vol{\Delta}<\infty$ precisely when $\Delta$ is
co-compact, i.e., $(G \times \ghat)/\Delta$ is compact.  A typical density result says that if
$\Delta$ is a uniform lattice and $\gaborG{g}$ is a frame for $L^2(G)$, then $\vol{\Delta} \le
1$. For separable uniform lattices $\Delta = \Lambda \times \Gamma \subset G \times \ghat$ this
result was proved by Gr\"o{}chenig in \cite{MR1601095}, and for non-separable uniform lattices (in
\emph{elementary} LCA groups) it is a consequence of results by Feichtinger and
Kozek~\cite{MR1601091}. Gr\"ochenig's proof is elementary using the Poisson summation formula, while
the argument for general lattices relies, as is often the case for results on non-separable
lattices, on the theory of pseudo differential operators.
We will give alternative proofs using only time-frequency analysis techniques. More importantly, we will
generalize density results to arbitrary closed subgroups $\Delta \subset G
\times\ghat$. 
We will show that:
\begin{enumerate}[(a)]
\item If $\gaborG{g}$ is a frame for $L^2(G)$, then $\vol{\Delta}<\infty$. 
\label{item:frame-implies-cocompact}
\item If $\Delta$ is a discrete subgroup and $\gaborG{g}$ is a frame,
  then $\vol{\Delta}\le 1$. \label{item:lattice-implies-less-1} 
\item $\{\pi(\nu)g\}_{\nu\in\Delta}$ is a Riesz family for $L^2(G)$ if, and only if, $\Delta$ is a
  uniform lattice with $\vol{\Delta}=1$ and $\{\pi(\nu)g\}_{\nu\in\Delta}$ is a frame for
  $L^2(G)$.  \label{item:riesz-iff-lattice}
\end{enumerate}
While \eqref{item:lattice-implies-less-1} might be expected, it is
rather surprising that density results can be formulated for
non-discrete Gabor systems as in \eqref{item:frame-implies-cocompact}
and \eqref{item:riesz-iff-lattice}. This extension relies crucially on the
fact that the new measure $\vol{\Delta}$ contains information on both
the subgroup $\Delta$ and its Haar measure. The forward direction in
\eqref{item:riesz-iff-lattice} is also somewhat unexpected. The
seemingly weak assumption that $\gaborG{g}$ is a Riesz family for some
closed subgroup $\Delta$ of the phase space has the strong conclusion
that $\Delta$ is a uniform lattice and that $\vol{\Delta}=1$.    
Moreover, we will see that in statement \eqref{item:frame-implies-cocompact} it is, in general, not
possible to be quantitative, that is, if $\Delta$ is non-discrete and co-compact, it will be
possible to construct a frame $\gaborG{g}$ regardless of the value of $\vol{\Delta}<\infty$. This
illustrates that the non-discrete case is rather different from the usual Gabor theory for
lattices. We will exhibit several of these differences in Section~\ref{sec:density} and
\ref{sec:duality}.


From our generalized density theorems, we are then able to extend the duality theory in Gabor
analysis to Gabor systems $\gaborG{g}$ with time-frequency shifts along arbitrary closed subgroups
$\Delta \subset G\times \ghat$. The most fundamental duality principle says that the Gabor system
$\gaborG{g}$ is a Parseval frame, \ie the system is associated with an isometric transform
$C_{g,\Delta}$, if and only if $\gaborG[\Delta^\circ]{\vol{\Delta}^{-1/2}g}$ is an orthonormal set,
where $\Delta^\circ$ denotes the adjoint of $\Delta$. 

We will prove two results that can be seen as an extension of this result. Firstly, to dual frames,
where one allows for two different window functions $g,h \in L^2(G)$ in the analysis and synthesis
transforms; this extension is known as the Wexler-Raz biorthogonality relations. Secondly to
non-tight frames; this result is simply known as the duality principle.  The Wexler-Raz
biorthogonality relations were previously available for non-separable, uniform lattices $\Delta
\subset G \times \ghat$ on elementary LCA groups $G=\R^n \times \T^\ell \times \Z^k \times F_m$ by
the work of Feichtinger and Kozek~\cite{MR1601091}, while the duality
principle (formulated without bounds) was proven by
Feichtinger and Zimmermann~\cite{MR1601107} for Gabor systems $\gaborG{g}$ in $L^2(\R^n)$ with
$\Delta$ being a non-separable, full-rank lattice in $\R^{2n}$. The
authors proved in~\cite{JakobsenCocompact2014} both the Wexler-Raz biorthogonality relations and the duality
principle on LCA groups for \emph{separable}, co-compact subgroups $\Delta = \Lambda \times \Gamma
\subset G \times \ghat$ using the theory of translation invariant systems; an approach that does not
generalize to the non-separable case.

Usually, the density/duality theory for \emph{non-separable} lattice Gabor systems relies on the
theory of pseudo-differential operators and von Neumann algebra techniques.  In particular, the
results of Feichtinger and Kozek~\cite{MR1601091} use concepts of function space Gelfand triples and
generalized Kohn-Nirenberg symbols.  
To cite from Gr\" ochenig's book~\cite{MR1843717}:
\begin{quote}
  These generalizations [\emph{density and duality results for \emph{non-separable} time-frequency
    lattices in the euclidean space}], however, require a completely different approach that involves the analysis of
  pseudo-differential operators with periodic symbols.
\end{quote}

The present paper provides density and duality theorems for Gabor systems $\gaborG{g}$ with time-frequency
shifts along (possibly non-separable) closed subgroups $\Delta \subset
G\times \ghat$ for general second countable
LCA groups $G$. In spite of the above comments, we are able to develop the theory solely within the setting of time-frequency analysis. Indeed,
our proofs are based on Weil's formula, the Fourier transform, the short-time Fourier transform, and
frame theory. 

We mention that duality results in the discrete case have been 
generalized in other directions; we refer the reader to
\cite{MR3319538,MR2441141,MR2535465,MR2078264,Fan2015} and the
references therein. Generalizations of the density theorem also exist;
in particular, Ramanathan and Steger~\cite{MR1325536} obtained density
results for $\gaborG{g}$ in $L^2(\R^n)$, where $\Delta$ is a discrete set,
but not necessarily a subgroup. 
We refer the reader to the survey paper by Heil~\cite{MR2313431} for a
detailed account of the history and evolution of density results in
Gabor analysis. For an introduction to Gabor analysis and frame
theory, we refer to \cite{MR1843717,MR1946982}. 

The paper is organized as follows: Section~\ref{sec:harmonic-ana} and \ref{sec:frame} contain
preliminary facts and results on Fourier analysis on LCA groups and frame theory, respectively. Some
new results on the non-existence of continuous Riesz families are included in
Section~\ref{sec:frame}; these results are essential for our
development in the later sections, however, they are also of independent interest. In Section~\ref{sec:Gabor} we introduce Gabor systems 
and show three key lemmas that will be
important in the proofs of the main results in Sections~\ref{sec:density} and \ref{sec:duality}. In these
sections we show density and duality results for Gabor frames $\gaborG{g}$, where $\Delta$ is a
closed subgroup of the time-frequency domain $G\times\ghat$. Appendix~\ref{sec:proof-lemma-xx}
contains results on the Feichtinger algebra $S_{0}$ that are needed for the proofs in Section~\ref{sec:duality}.






\section{Harmonic analysis on LCA groups}
\label{sec:harmonic-ana}

We let $G$ denote a second countable locally compact abelian group. To $G$ we associate
its dual group $\ghat$ which consists of all characters, \ie all continuous homomorphisms from $G$
into the torus $\T \cong \setprop{z\in\C}{ \abs{z} =1}$. Under pointwise multiplication $\ghat$ is
also a locally compact abelian group. Throughout the paper we use multiplication as
group operation in $G$, $\ghat$, and $G \times \ghat$, and we denote
the identity element by $e$. By the Pontryagin duality theorem, the dual group
of $\ghat$ is isomorphic to $G$ as a topological group, \ie $\ghhat \cong G$. 

We denote the Haar measure on $G$ by $\mu_G$. The (left) Haar measure on any locally compact group
is unique up to a positive constant. From $\mu_G$ we define $L^1(G)$ and the Hilbert space $L^2(G)$
over the complex field in the usual way. Since $G$ is assumed to be second
countable, these function spaces are separable. We define the Fourier transform of $f\in L^1(G)$ by
\[ 
\cF f(\omega) = \hat{f}(\omega) = \int_{G} f(x) \, \overline{\omega(x)} \, d\mu_G(x), \quad
\omega \in \ghat.
\] 
If $f\in L^1(G), \hat{f} \in L^1(\ghat)$, and the measures on $G$ and $\ghat$ are normalized so that
the Plancherel theorem holds (see \cite[(31.1)]{MR0262773}), then the function $f$ can be recovered from
$\hat{f}$ by the inverse Fourier transform
\[
f(x) = \cF^{-1}\hat{f}(x) = \int_{\ghat} \hat{f}(\omega) \, \omega(x) \, d\mu_{\ghat}(\omega),
\quad a.e.\ \ x\in G.
\]
If, in addition, $f$ is continuous, the inversion formula holds pointwise. We assume that the
measure on a group $\mu_G$ and the measure on its dual group $\mu_{\ghat}$ are normalized this way,
and we refer to them as \emph{dual measures}. Under this convention, the Fourier transform $\cF$ is
an isometric isomorphism between $L^2(G)$ and $L^2(\ghat)$.


For $\nu = ( \lambda,\gamma) \in G\times \ghat$, we let $\pi(\nu)$ denote the time-frequency shift
operator $E_{\gamma}T_{\lambda}$. 
It is clear that $\pi(\nu)$ is a unitary operator on $L^2(G)$. The
commutator relation
\begin{align}
T_{\lambda}E_{\gamma} & = \overline{\gamma(\lambda)}
E_{\gamma}T_{\lambda} \nonumber 
\intertext{leads to the following useful identities:}
\pi(\nu)^* & = \overline{\gamma(\lambda)} \, \pi(\nu^{-1}),
\label{eq:cr2} 
\\  \pi(\nu_1)\pi(\nu_2) & = \overline{\gamma_2(\lambda_1)} \, \pi(\nu_1\nu_2) \label{eq:cr3}
\\
\pi(\nu_1)\pi(\nu_2) & = \gamma_1(\lambda_2) \overline{\gamma_2(\lambda_1)} \, \pi(\nu_2)\pi(\nu_1), \label{eq:cr4}
\end{align}
where $\nu_i=(\gamma_i,\lambda_i)$, $i=1,2$, and $\pi(\nu)^*$ denotes
the adjoint operator of $\pi(\nu)$. 


We let $\Delta$ denote a closed subgroup of $G\times\ghat$ with measure $\mu_\Delta$. To ease
notation, when the measure is clear from the context, we write $d\nu$ in place of $d\mu_\Delta(\nu)$
and likewise for other measures. In our settings Weil's formula will relate integrable functions
over $G\times \ghat$ with integrable functions on the quotient space $(G\times \ghat)/\Delta$, where
$\Delta$ is a closed subgroup of $G\times \ghat$. Let $c_{\Delta}: G\times\ghat \to
(G\times\ghat)/\Delta, \ c_{\Delta}(\chi) = \chi\Delta$ be the \emph{canonical map} from
$G\times\ghat$ onto $(G\times\ghat)/\Delta$. If $f\in L^1(G\times\ghat)$, then the function $\dot
\chi \mapsto \int_{\Delta} f(\chi\nu) \, d\nu$ with $\dot \chi = c_{\Delta}(\chi)$, defined almost
everywhere on $(G\times\ghat)/\Delta$, is integrable. Furthermore, when two out of the three Haar
measures on $G\times\ghat$, $\Delta$ and $(G\times\ghat)/\Delta$ are given, the third can be
normalized in a unique way so that \emph{Weil's formula}
\begin{equation}
\label{eq:weil-formula} 
\int_{G\times\ghat} f(\chi) \, d\chi = \int_{(G\times\ghat)/\Delta} \int_{\Delta} f(\chi\nu) \, d\nu \, d\dot \chi 
\end{equation}
holds. 

The annihilator group $\Delta^{\perp}$ of $\Delta \subset G\times\ghat$ is given by
\[ 
\Delta^{\perp}  = \setprop{ (\beta,\alpha) \in \ghat\times G}{\gamma(\alpha) \beta(\lambda) = 1 \
  \text{for all} \ \nu=(\lambda,\gamma) \in \Delta}.
\] 
The annihilator is a closed subgroup of $\ghat\times G$. Moreover,
\[ 
\widehat \Delta \cong (\ghat\times G) / \Delta^{\perp} \quad \text{ and } \quad
((G\times\ghat)/\Delta)^{\widehat{ \ \ \, }} \cong \Delta^{\perp}.
\]
These relations show that for the closed subgroup $\Delta$ the quotient $(G\times \ghat)/\Delta$ is
compact if, and only if, $\Delta^{\perp}$ is discrete. 
  Finally, we define the adjoint $\Delta^{\circ}$ of $\Delta \subset G\times\ghat$
as 
\[ 
\Delta^{\circ} := \{ \mu \in G \times \ghat \, : \, \pi(\mu)\pi(\nu)= \pi(\nu)\pi(\mu) \ \ \forall
\nu\in\Delta\}.
\] 
The annihilator and adjoint of a closed subgroup $\Delta$ are identical up to a change of
coordinates. To see this, we introduce the mapping
\[ 
\Phi: G\times \ghat \ \to \ \ghat\times G, \ \Phi (x,\omega) = (\omega,x) \ \text{ for } \
(x,\omega)\in G\times\ghat.
\]
It is clear that $\Phi$ is a measure preserving, topological group isomorphism. 
It follows from \eqref{eq:cr4} that  $\Phi(\Delta^{\circ}) =
\Delta^{\perp}$. 


We will use the following \textbf{general setup}. We assume a Haar measure on $G$. On the dual group
of any LCA group, we assume the dual measure (such that the Plancherel theorem holds). Furthermore,
we assume a Haar measure on the closed subgroup $\Delta$ of $G\times\ghat$.  By requiring that
Weil's formula~\eqref{eq:weil-formula} holds, there is a uniquely determined measure
$\mu_{(G\times\ghat)/\Delta}$ on $(G\times\ghat)/\Delta$. From this measure, we define the
\emph{size} of the subgroup $\Delta$ as
 \[ 
  \vol{\Delta} = \mu_{(G\times\ghat)/\Delta}((G\times\ghat)/\Delta) .
 \]
Intuitively, small values of $\vol{\Delta}$ suggest that $\Delta$ is ``dense'', while large values of
$\vol{\Delta}$ suggest that $\Delta$ is ``sparse''.
 \begin{remark} \label{rem:group-size}
   \begin{enumerate}[(i)]
   \item In case $\Delta$ is co-compact, by the Plancherel identity, we see that the discrete group
     $\Delta^{\perp}$ is equipped with the Haar measure $\vol{\Delta}^{-1} \mu_c$, where $\mu_c$ is
     the counting measure. In particular, the canonical choice $\vol{\Delta}=1$ comes from the
     probability measure of $(G \times \ghat)/\Delta$ or, equivalently, the counting measure on
     $\Delta^\perp$.
   \item In case $\Delta$ is a discrete, co-compact subgroup, i.e., a uniform lattice, then
     $\vol{\Delta}$ is closely related to the \emph{lattice size} of $\Delta$.  Let
     $s(\Delta)=\mu_{G \times \ghat}(X)$, where $X$ is a Borel section, also called a fundamental
     domain, of $\Delta$ in $G \times \ghat$ \cite{MR1434478,MR1601095}. Now, if we equip $\Delta$
     with the counting measure, then $s(\Delta)=\vol{\Delta}$. Especially for $\R^n$, if $\Delta = P
     \Z^{2n}, \, P\in \mathrm{GL}_\R(2n)$, then $\vol{\Delta} = \vert \! \det(P) \vert$.
   \end{enumerate}
 \end{remark}

\begin{lemma} \label{le:adjoint-lattice} Let $\Delta$ be a closed subgroup of $G\times \ghat$. Then the following holds:
\begin{enumerate}[(i)]
\item $\vol{\Delta} < \infty$ if, and only if, $\Delta$ is co-compact,
\item $\vol{\Delta^{\perp}} < \infty$ if, and only if, $\vol{\Delta^{\circ}} < \infty$ if, and only if, $\Delta$ is discrete.
\end{enumerate}
Furthermore, if $\Delta$ is discrete and co-compact, then
\begin{enumerate}[(i)]
\item [(iii)] $\Delta^{\perp}\subset \ghat\times G$ and $\Delta^{\circ}\subset G\times\ghat$ are discrete and co-compact subgroups,
\item [(iv)] $\vol{\Delta} \vol{\Delta^{\perp}} = 1$ and $\, \vol{\Delta} \vol{\Delta^{\circ}} = 1$.
\end{enumerate}
\end{lemma}
\begin{proof} Statement~(i) is just a reformulation of the fact that the Haar measure of an LCA
  group is finite if, and only if, the group is compact. Since $\Delta^{\perp}$ is discrete if, and
  only if $(G\times\ghat)/\Delta$ is co-compact, statement~(ii) follows from~(i). Statements~(iii)
  and~(iv) for $\Delta^\perp$ can be found in \cite{MR1601095}. The statements for $\Delta^{\circ}$
  follow by the relationship between $\Delta^{\circ}$ and $\Delta^{\perp}$, see also \cite[Lemma
  7.7.4]{MR1601091}.
\end{proof}




\section{Frame theory}
\label{sec:frame}

We need a rather general variant of frames, usually called
continuous frames, introduced by Ali, Antoine, and Gazeau
\cite{MR1206084} and Kaiser~\cite{MR1287849}.
\begin{definition} 
\label{def:cont-frames} 
Let $\cH$ be a complex Hilbert space, and let $(M,\Sigma_M,\cfm)$ be a measure space, where
$\Sigma_M$ denotes the $\sigma$-algebra and $\cfm$ the non-negative measure. A family of vectors
$\set{f_k}_{k \in M}$ in $\cH$ is a \emph{frame} for $\cK:=\overline{\Span}\set{f_k}_{k \in
  M}$ with respect to $(M,\Sigma_M,\cfm)$ if
  \begin{enumerate}[(a)]
  \item $k\mapsto f_k$ is weakly measurable, \ie for all $f \in \cK$,
    the mapping $M \to \C, k \mapsto \innerprod{f}{f_k}$ is
    measurable, and
  \item there exist constants $A,B>0$ such that
    \begin{equation}
      \label{eq:cont-frame-inequality} 
      A \norm{f}^2 \le \int_M \abs{\innerprod{f}{f_k}}^2 d\cfm(k) 
      \le B \norm{f}^2 \quad \text{for all } f \in \cK. 
    \end{equation}
  \end{enumerate}
The constants $A$ and $B$ are called \emph{frame bounds}.
\end{definition}

When $\set{f_k}_{k\in M}$ is a frame for its closed linear span $\cK$, we say that
$\set{f_k}_{k\in M}$ is a \emph{basic} frame. If $\cK=\cH$, we say $\set{f_k}_{k\in M}$ is \emph{total}.
If $\set{f_k}_{k\in M}$ is weakly measurable and the upper bound in the
inequality~(\ref{eq:cont-frame-inequality}) holds, then $\set{f_k}_{k\in M}$ is a
\emph{Bessel family} with constant $B$. A frame $\set{f_k}_{k\in M}$ is said to be \emph{tight} if
we can choose $A = B$; if, furthermore, $A = B = 1$, i.e., $C_F$ is isometric, then $\set{f_k}_{k\in
  M}$ is a \emph{Parseval frame}.

To a Bessel family $F:=\set{f_k}_{k\in M}$ for $\cK \subset \cH$, we associate the \emph{analysis operator}
$C_F$ given by
\[ 
C_F: \cH \to L^2(M,\cfm), \quad f \mapsto (k \mapsto \innerprod{f}{f_k}).
\] 
The frame condition~\eqref{eq:cont-frame-inequality} simply says that this operator on $\cK$ is bounded below and above by $\sqrt{A}$
and $\sqrt{B}$, respectively, hence $C_F\vert_{\cK}$ is an injective, bounded linear operator with
closed range. The adjoint $D_F$ of $C_F$ is the so-called \emph{synthesis operator}; it is given
weakly by
\[
D_F: L^2(M,\cfm) \to \cH, \quad c \mapsto  \int_M c(k)
f_k \, d\cfm (k).
\]

For two Bessel families $F=\set{f_k}_{k\in M}$ and $G=\set{g_k}_{k\in
  M}$ we define the \emph{mixed frame
operator} $S_{F,G}=D_F C_G$. If $F$ is a frame, the \emph{frame operator} $S_F:=S_{F,F}$ is a bounded,
invertible, self-adjoint and positive operator. The Bessel families $F$ and $G$ are said to be \emph{dual frames} for $\cH$ if $S_{F,G} = \id[\cH]$, i.e.,
\begin{equation}
  \label{eq:cont-dual-weak}
  \innerprod{f}{g} = \int_M \innerprod{f}{g_k} \innerprod{f_k}{g} d\cfm (k) \quad \text{for all } f,g \in \cH.
\end{equation}
In this case we say that the following assignment
\begin{equation*}
 f = \int_M \innerprod{f}{g_k}f_k \, d\cfm (k) \quad \text{for $f \in \cH$}, 
\end{equation*}
holds in the weak sense. Dual frames for subspaces $\cK$ of $\cH$ are
defined similarly. Two dual frames are indeed frames for $\cH$, see, e.g.,
\cite{JakobsenReproducing2014}. On the other hand, given a frame $F=\set{f_k}_{k\in M}$ for $\cH$ one
can always find at least one dual frame; the canonical choice is
$\set{S_F^{\,-1} f_k}_{k\in M}$. The following result is well-known in
frame theory.
\begin{theorem} 
\label{thm:A} 
Let $\mathcal{H}$ be a Hilbert space, and let $A,B>0$. Then the following statements are equivalent:
\begin{enumerate}[(i)]
\item $\{f_k\}_{k\in M}$ is a frame for $\mathcal{H}$ with bounds $A$ and $B$;
\item $\{f_k\}_{k\in M}$ is a Bessel family in $\mathcal{H}$ with bound $B$ and there exists another Bessel family $\{g_k\}_{k\in M}$ in $\mathcal{H}$ with bound $A^{-1}$ such that \eqref{eq:cont-dual-weak} holds.
\end{enumerate}
\end{theorem}

Frames as defined in Definition~\ref{def:cont-frames} are often called \emph{continuous} frames,
with the notion \emph{discrete} frames reserved for the case, where $M$ is countable and $\cfm$ is
the counting measure. We will not adapt this terminology. A family of vectors $\{f_k\}_{k \in M}$ will be called \emph{continuous}
if $\cfm$ is non-atomic and \emph{discrete} if $\cfm$ is purely atomic \emph{on $\sigma$-finite subsets}. Recall that a set $E \in
\Sigma$ of positive measure is an \emph{atom} if for any measurable subset $F$ of $E$ either
$\cfm(F)=0$ or $\cfm(E \setminus F)=0$. A measure is called \emph{purely atomic} if every
measurable set of positive measure contains an atom and \emph{non-atomic} if there are no
atoms. Every measure can be uniquely decomposed as a sum of a purely atomic and a non-atomic measure
in the sense of Johnson~\cite{MR0279266}.

Let us explain our terminology of discrete and continuous frames. For $a \in L^2(M,\cfm)$ the
support $K:=\supp a$ is $\sigma$-finite, hence we can write $K=\cup_{i \in I} M_i \cup N$, where
each $M_i$ is an atom of finite measure, $I$ is at most countable, and $N$ is an atomless measurable
set. Assume first $\cfm$ is atomic whenever restricted to the subalgebra $\Sigma_K=\setprop{E \cap
  K}{E \in \Sigma}$, where $K$ is a $\sigma$-finite set. Then $\cfm(N)=0$. Since we are interested
in $L^2$-functions, we will either tacitly ignore such null sets or simply say we have equality up to
sets of measure zero. Functions, or rather equivalence classes of functions, in $L^2(M,\cfm)$ are
constant on every atom of $M$, hence any frame $k \mapsto f_k$ will also be constant on atoms (up to
sets of measure zero). Moreover, $a \in L^2(M,\cfm)$ is of the form $\sum_{i \in I} a_i \, \mathds{1}_{M_i}$
for some coefficients $a_i \in \C$. Let $f_i$ denote the value of $k \mapsto f_k$ on $M_i$. Then the
synthesis operator on $a \in L^2(M,\cfm)$ is
\begin{equation}
D_{F} a = \sum_{i \in I} \cfm(M_i) a_i f_i  \qquad \text{for
  all $f \in \cH$},
\label{eq:atomic-synthesis-op}
\end{equation}
which is indeed a discrete representation. Assume, on the other hand, that $\cfm$ is non-atomic on
each $\sigma$-finite subset $K$ of $M$. In the decomposition $K=\cup_{i \in I} M_i \cup N$, we now
have $K=N$ since $K$ is atomless. We can write $K=\cup_{i \in I} K_i$, where $K_i$ is of finite
measure. Then, by a classical result of Sierpinski~\cite[Lemma~52.$\alpha$]{MR0222234}, the measure $\cfm$ takes
a continuum of values $\itvcc{0}{\cfm{(K_i)}}$ on the measurable subsets of $K_i$, which justifies the name \emph{continuous frame}.

If $F=\set{f_k}_{k \in M}$ is a basic frame with bounds $A$ and $B$ and if $C_F$ has dense range,
then $C_F\vert_{\cK}$ is invertible on all of $L^2(M,\cfm)$, and we say that $\set{f_k}_{k \in M}$ is a
\emph{basic Riesz family} with bounds $A$ and $B$. Equivalently, one can define basic Riesz families with
bounds $A$ and $B$ as families of vectors $F=\set{f_k}_{k \in M}$ for which $D_F$ defined on simple,
integrable functions is bounded below and above by $\sqrt{A}$ and $\sqrt{B}$, respectively, i.e.,
\[ 
A \norm[L^2(M)]{a}^2 \le \norm{\int_M a(k)
f_k \, d\cfm (k)}^2 \le B \norm[L^2(M)]{a}^2
\]   
for all simple functions $a$ on $M$ with finite support. If $M$ is countable and equipped with the
counting measure, a basic Riesz family is simply a Riesz basis for its
closed linear span, also called a Riesz sequence. 

Our notion of discrete frames might appear overly technical compared
to the usual definition (i.e., $M$ countable and $\mu_M$ the counting
measure). However, it allows us to classify the following two
pathological ``continuous'' examples as discrete frames.
 \begin{example}
\label{exa:continuous-riesz}
 Consider the following two examples with a ``continuous'' index set $M=\R$:
  \begin{enumerate}[(a)]
  \item Let $\cH = \ell^2(\Z)$, let $\set{e_k}_{k \in \Z}$ be its
    standard orthonormal basis, and equip $M=\R$ with a purely
    atomic measure, whose atoms are the intervals $\itvco{n}{n+1}$, $n
    \in \Z$, each with measure $1$. Define
    $\{f_k\}_{k \in M} \subset \cH$ by $f_k=e_{\floor{k}}$, $k
    \in \R$.
 \item Let $\cH = L^2(\itvcc{0}{1})$, and let $M=\R$. Fix $a \in
   \R$ and define $\mu_M=\sum_{n \in \Z} \delta_{n+a}$, where
   $\delta_x$ denotes the Dirac measure at $x \in \R$. Define
   $\{f_k\}_{k \in M} \subset \cH$ by $f_k(x)=\myexp{2\pi ikx}$.  
 \end{enumerate}
It is not difficult to show that both in case (a)
and (b) the family $\{f_k\}_{k \in M}$  is a Parseval frame; it is even a Riesz family. Since the
measure in both cases is purely atomic, the frame $\{f_k\}_{k \in M}$ is
said to be discrete.  
 \end{example}

There has recently been some interest in the study of (continuous) Riesz families
\cite{Arefijamaal2014,MR3066577}, also called Riesz-type frames in
\cite{MR1968116}. Example \ref{exa:continuous-riesz}(b) is a concrete
version of \cite[Proposition 3.7]{Arefijamaal2014}.  The following
result shows, however, that this concept brings little new to the well-studied subject of
(discrete) Riesz sequences. To be more concise, the result shows that norm bounded, basic Riesz
families necessarily are discrete. 
\begin{proposition}
 \label{th:rieszseq-implies-discrete-1} 
 Let $(M,\Sigma,\cfm)$ be a measure space. Suppose that $\{f_k\}_{k\in
   M}$ is a basic Riesz family in
 $\cH$ that is essentially norm bounded, i.e., $C:=\sup_{k \in
   M}\norm{f_k}<\infty$. Then $\{f_k\}_{k\in M}$ is a discrete family. Furthermore, it holds that
\begin{equation}
\label{eq:mesure-bb}
 \inf_{E \in \Sigma_0}\cfm(E)>0, \qquad \text{where } \Sigma_0 =
\setprop{E \in \Sigma}{\cfm(E)>0}.
\end{equation}
\end{proposition}
\begin{proof} 
Let $a$ be an integrable simple function on $M$.
From the computation
\[
\norm{\int_M a(k) f_k dm(k)} \le \int_M \abs{a(k)} \norm{f_k} dm(k)
\le C \int_M \abs{a(k)} dm(k),
\]
we see that the lower Riesz bound implies 
\[ 
\sqrt{A} \norm[L^2(M)]{a} \le C \norm[L^1(M)]{a}.
\]
For $a=\mathds{1}_{E}$ with $E \in \Sigma$ and $0 <\cfm{(E)} < \infty$, this
in turn implies that
\[ \sqrt{A} \, \cfm{(E)}^{1/2} \le C \cfm{(E)}, \]
and hence $\cfm(E)\ge A/C^2$. This shows the furthermore-part.

Let $K=\cup_{i \in I} K_i$ be any $\sigma$-finite set, where each $K_i$ is of finite measure. We need to
show that $\cfm$ restricted to the subalgebra $\Sigma_K=\setprop{E \cap K}{E \in \Sigma}$ is purely
atomic. Suppose on the contrary that it is not. Then there is an atomless set $N$ of positive
measure in $\Sigma_K$. For some $i_0 \in I$ the intersection $N \cap K_{i_0}$ has positive
measure. Clearly, $N_0:=N \cap K_{i_0}$ is also atomless, hence we can split this set into two sets of
positive measure. The smallest in measure of these two sets, say $N_1$, is of measure $\cfm(N_1)\le
\cfm(N_0)/2$. Continuing this way we obtain sets of arbitrarily small measure, contradicting~\eqref{eq:mesure-bb}.
\end{proof}



From \eqref{eq:mesure-bb} we see that for norm-bounded Riesz families, the atoms $M_i$ in the
representation \eqref{eq:atomic-synthesis-op} are bounded from below in measure. Hence, even if we
consider the sum in~\eqref{eq:atomic-synthesis-op} as a Riemann type sum, there is a bound to any
refinement.

If we assume that $M$ is a Hausdorff topological group and that $\cfm$ satisfies certain weak regularity
assumptions, \eg $M$ being a locally compact group with the usual left Haar measure, then the existence of a norm-bounded Riesz family
forces the group $M$ to be discrete.
  
\begin{proposition} \label{th:rieszseq-implies-discrete} 
Let $M$ be a Hausdorff topological group with a left Haar measure
$\cfm$ (as defined by Fremlin~\cite[Def.~441D]{MR2462372}). If
$\{f_k\}_{k\in M}$ is a norm bounded basic Riesz family, then $M$ is a discrete group.
\end{proposition}
\begin{proof}
From Proposition~443O in \cite{MR2462372} we know that $\cfm$ is \emph{not} non-atomic if and only if there is the
discrete topology on $M$. However, by Proposition~\ref{th:rieszseq-implies-discrete-1}, the measure
$\cfm$ is clearly not non-atomic, thus the result follows. 
\end{proof}

We end this section with a Riesz family variant of
Theorem~\ref{thm:A}. 
\begin{theorem}[\!\!\cite{JakobsenCocompact2014}] 
\label{thm:B}
Let $\mathcal{H}$ be a Hilbert space, let $A,B>0$, and let $M$ be a countable index
set equipped with the counting measure. Then the following statements are equivalent:
\begin{enumerate}[(i)]
\item $\{f_k\}_{k\in M}$ is a basic Riesz family (i.e., a Riesz sequence) in $\mathcal{H}$ with bounds $A$ and $B$;
\item $\{f_k\}_{k\in M}$ is a Bessel family $\mathcal{H}$ with bound $B$ and there exists a Bessel family $\{g_k\}_{k\in M}$ in $\mathcal{H}$ with bound $A^{-1}$ such that $\langle f_k , g_\ell\rangle = \delta_{k,\ell}$, $k,\ell\in M$. 
\end{enumerate}
\end{theorem}


\section{Gabor systems}
\label{sec:Gabor}

The Gabor system $\gaborG{g}=\gabor{g}$ is \emph{regular} when $\Delta$ is a closed subgroup of
$G\times\ghat$. If $\Delta$ is not a
subgroup, e.g., merely a set of points, the Gabor system is \emph{irregular}. If $\Delta = \Lambda
\times \Gamma$ for closed subgroups $\Lambda\subset G$ and $\Gamma\subset \ghat$, we say that
$\gaborG{g} = \{E_{\gamma}T_{\lambda}g\}_{\lambda\in\Lambda,\gamma\in\Gamma}$ is a \emph{separable}
Gabor system. If $\Delta$ is not assumed to have this form, $\gaborG{g}$ is \emph{non-separable}. In
this work we shall consider non-separable, regular Gabor systems.

The analysis, synthesis, and the (mixed) frame operator for Gabor Bessel systems are defined as in
Section~\ref{sec:frame}. In particular, the (mixed) frame operator for two Gabor Bessel systems
generated by the functions $g,h\in L^2(G)$ takes the form
\begin{equation*}
S_{g,h} : L^2(G)\to L^2(G), \quad S_{g,h} = \int_{\Delta} \langle \, \cdot \, , \pi(\nu)g \rangle \pi(\nu)h \,  d\nu .
\end{equation*}
If $g=h$, we recover the frame operator $S_g=S_{g,g}$, also simply denoted by $S$.

It is straightforward to show that the frame operator commutes with
time-frequency shifts with respect to the group $\Delta$. 
\begin{lemma}
\label{thm:frame-op-commutes-time-freq}
Let $g,h\in
L^2(G)$, $\Delta\subset G\times\ghat$, and let $\gaborG{g}$ and $\gaborG{h}$ be Bessel systems. If~$\Delta$~is a closed subgroup of $G \times \ghat$, then the following holds:
\begin{enumerate}[(i)]
\item $S_{g,h} \pi(\nu) = \pi(\nu) S_{g,h}$ for all $\nu \in \Delta$,
\item If $\gabor{g}$ is a frame, then
\[ S^{-1} \pi(\nu) = \pi(\nu) S^{-1} \quad \text{and} \quad S^{-1/2}
\pi(\nu) = \pi(\nu) S^{-1/2} \quad \text{for all } \nu \in \Delta. \]
\end{enumerate}
\end{lemma} 

\begin{remark}
\label{re:canonical-gabor-frames}
  Lemma~\ref{thm:frame-op-commutes-time-freq} implies that the
  canonical dual frame of a Gabor frame again is a Gabor system $\gaborG{S^{-1}g}$
  and that the Gabor system $\gaborG{S^{-1/2}g}$ is
  a Parseval frame. In particular, if $\gaborG{g}$ is a Riesz basis, then $\gaborG{S^{-1/2}g}$ is an
  orthonormal basis.
\end{remark}



We are interested in those pairs $(g,\Delta)\subset(L^2(G),G\times \ghat)$ for which $\gaborG{g}$ is
a Gabor frame for $L^2(G)$, that is, closed subgroups $\Delta \subset G \times \ghat$ and window functions $g\in L^2(G)$ for which there exists constants $0<A\le B<\infty$ such that 
\begin{equation*} 
A \, \Vert f \Vert^2 \le \int_{\Delta} \vert \langle f, \pi(\nu) g\rangle\vert^2 \, d\nu \le B \, \Vert f \Vert^2
\end{equation*}
for all $f \in L^2(G)$.

We will need the following well-known Plancherel theorem for the short-time Fourier transform $\mathcal{V}_{g}:=C_{g,G\times \ghat}$.
\begin{lemma}[\!\!\cite{MR1601095}]
 \label{le:stft-norm-preserving} 
For $f_1,f_2,g,h\in L^2(G)$ 
the following assertions are true:
\begin{enumerate}[(i)]
\item $\mathcal{V}_{g}f\in L^2(G\times \ghat)$ and $
\Vert \mathcal{V}_{g}f\Vert_{L^2(G\times \ghat)}^2 = \int_{G\times\ghat} \vert \langle f, \pi(\nu) g\rangle \vert^2 \, d\nu = \Vert g \Vert^2 \, \Vert f \Vert^2$,
\item \begin{equation} \label{eq:dual-stft-norm-bound} \int_{G\times\ghat}  \big\vert\langle f_1, \pi(\nu) g\rangle \langle \pi(\nu)h, f_2\rangle \big\vert \, d\nu  \le \Vert f_1 \Vert \, \Vert f_2 \Vert \, \Vert g \Vert \, \Vert h \Vert < \infty,\end{equation}
\item \begin{equation}
\int_{G\times\ghat} \langle f_1, \pi(\nu) g\rangle \langle \pi(\nu)h, f_2\rangle \, d\nu = \langle f_1,f_2\rangle \langle h,g\rangle.\label{eq:STFT-dual}
\end{equation}
\end{enumerate}
\end{lemma}
\begin{proof}
Statement (i) and (iii) can be found in \cite{MR1601095}.
The inequality~\eqref{eq:dual-stft-norm-bound} follows directly from the following computation:
\begin{align*}
\int_{G\times\ghat}  \big\vert\langle f_1, \pi(\nu) g\rangle \langle \pi(\nu)h, f_2\rangle \big\vert \, d\nu & \le \Big( \int_{G\times\ghat} \big\vert\langle f_1, \pi(\nu) g\rangle \big\vert^2  \, d\nu \Big)^{1/2} \Big( \int_{G\times\ghat} \big\vert\langle f_2, \pi(\nu) h\rangle \big\vert^2  \, d\nu \Big)^{1/2} \\
& = \Vert f_1 \Vert \, \Vert f_2 \Vert \, \Vert g \Vert \, \Vert h \Vert.
\end{align*}
\end{proof}

Lemma~\ref{le:stft-norm-preserving} shows that for any non-zero function $g\in L^2(G)$ the system
$\gaborG[G\times \ghat]{g}$ is a tight frame with bound $A=\Vert g \Vert^2$. More generally, if $\langle g,h\rangle \ne 0$, then $\gaborG[G\times\ghat]{g}$ and $\gaborG[G\times\ghat]{h}$ are dual frames, and we have a (weak) reproducing formula:
\[
f= \frac{1}{\langle h,g\rangle}\int_{G\times\ghat} \langle f, \pi(\nu) g\rangle \pi(\nu)h  \, d\nu  \quad\text{for all } f\in L^2(G).
\]

\subsection{Three key lemmas} 
\label{sec:key-identity}

In this subsection we prove three observations that will be important in the subsequent sections. 
 
\begin{lemma} 
\label{le:wex-raz-proof-inner-product-rewrite} 
Let $\chi = (x,\omega)\in (G\times \ghat)$ and $\mu=(\alpha,\beta)\in
\Delta^{\circ}\subset G\times\ghat$. If $\Delta$ is a closed subgroup of $G\times\ghat$, then the equality
\[ 
\langle h, \pi(\mu) g\rangle \langle \pi(\mu) f_1,f_2\rangle = \int_{(G\times \ghat)/\Delta}
\overline{\omega(\alpha)} \beta(x) \, \int_{\Delta} \langle \pi(\chi)^* f_1, \pi(\nu)g\rangle\langle
\pi(\nu) h, \pi(\chi)^* f_2\rangle \, d\nu \, d\dot\chi
\]
holds for all $f_1,f_2,g,h\in L^2(G)$.
\end{lemma}
\begin{proof} 
  If $\mu\in\Delta^{\circ}$, then by Lemma~\ref{le:stft-norm-preserving}(iii) we have that, for
  $f_1,f_2,g,h\in L^2(G)$,
  \[
  \langle h,\pi(\mu)g\rangle\langle \pi(\mu) f_1,f_2\rangle = \int_{G\times\ghat} \langle \pi(\mu)
  f_1, \pi(\chi) \pi(\mu) g\rangle \langle \pi(\chi) h , f_2\rangle \, d\chi.
  \]
By Weil's formula for the closed subgroup $\Delta$, the above equality becomes
  \[ 
  \langle h,\pi(\mu)g\rangle\langle \pi(\mu) f_1,f_2\rangle = \int_{(G\times\ghat)/\Delta}
  \int_{\Delta} \langle \pi(\mu) f_1, \pi(\chi\nu) \pi(\mu) g\rangle \langle \pi(\chi\nu) h ,
  f_2\rangle \, d\nu \, d\dot\chi.
  \] 
  For $\chi=(x,\omega)\in G \times \ghat$ and $\mu=(\alpha,\beta)\in\Delta^{\circ}\subset G \times \ghat$ we have
  \[ 
  \langle \pi(\mu) f_1, \pi(\chi\nu) \pi(\mu) g\rangle \langle \pi(\chi\nu) h , f_2\rangle =
  \overline{\omega(\alpha)} \beta(x) \langle \pi(\chi)^* f_1, \pi(\nu) g\rangle \langle \pi(\nu) h ,
  \pi(\chi)^* f_2\rangle,
  \] 
  which follows from \eqref{eq:cr3}, \eqref{eq:cr4}, and $\pi(\nu)\pi(\mu) = \pi(\mu)\pi(\nu)$ for
  $\nu\in\Delta, \mu\in\Delta^{\circ}$. This completes the proof.
\end{proof}

\begin{lemma} 
\label{th:biorth-identity} 
Let $\Delta$ be a closed, \emph{co-compact} subgroup of $G\times\ghat$, and let $g,h\in L^2(G)$. If $\gaborG{g}$ and
$\gaborG{h}$ are dual frames, then
\[ 
\langle h,\pi(\mu)g\rangle = \vol{\Delta} \, \delta_{\mu,e} \quad \text{for all } \mu\in
\Delta^{\circ}.
\]
\end{lemma}

\begin{proof}
  Let $\mu=(\alpha,\beta)\in \Delta^{\circ}$ and take $f \in L^2(G)$.  By
  Lemma~\ref{le:wex-raz-proof-inner-product-rewrite} we have that
  \[
  \langle h, \pi(\mu) g\rangle \langle \pi(\mu) f,f\rangle = \int_{(G\times \ghat)/\Delta}
  \overline{\omega(\alpha)} \beta(x) \, \int_{\Delta} \langle \pi(\chi)^* f, \pi(\nu)g\rangle\langle
  \pi(\nu) h, \pi(\chi)^* f\rangle \, d\nu \, d\dot\chi.
  \] 
  Since $\{\pi(\nu)g\}_{\nu\in\Delta}$ and $\{\pi(\nu)h\}_{\nu\in\Delta}$ are dual frames by
  assumption, this equation simplifies to
  \begin{align*}
    \langle h, \pi(\mu) g\rangle \langle \pi(\mu) f,f\rangle &= \int_{(G\times \ghat)/\Delta}
    \overline{\omega(\alpha)} \beta(x) \, \langle \pi(\chi)^* f, \pi(\chi)^* f\rangle \, d\dot\chi
    \\ &=
    \int_{(G\times \ghat)/\Delta} \overline{\omega(\alpha)} \beta(x) \, d\dot\chi \, \langle f,
    f\rangle.
  \end{align*}

  The function $\chi\mapsto \mu(\chi) := \overline{\omega(\alpha)} \beta(x)$ is continuous on the
  compact domain $(G\times\ghat)/\Delta$, and it satisfies $\mu(\chi_1\chi_2) =
  \mu(\chi_1)\mu(\chi_2)$. Therefore, \cite[Lemma~23.19]{MR0156915}
implies that, for all
  $f\in L^2(G)$,
  \[
  \langle h,\pi(\mu)g\rangle\langle \pi(\mu) f,f\rangle = \begin{cases} \vol{\Delta} \, \langle
    f,f\rangle & \text{if } \mu=e, \\ 0 & \text{if }\mu\ne e.
  \end{cases}
  \]
  It follows that $\langle h,\pi(\mu)g\rangle = \vol{\Delta} \delta_{\mu,e}$ for $\mu\in
  \Delta^{\circ}$.
\end{proof}

\begin{lemma} 
\label{pr:gabor-fundamental-function-bounded} 
Let $\Delta$ be a closed subgroup of $G\times\ghat$, and let $f_1,
f_2, g,h \in L^2(G)$. If $\gaborG{g}$ and $\gaborG{h}$ are Bessel systems with Bessel bound $B_g$ and
$B_h$, respectively, then for fixed $f_1$ and $f_2$, the mapping
\[ 
\varphi: \ G\times \ghat\to \C , \ \ \chi \mapsto \int_{\Delta} \langle \pi(\chi)^* f_1, \pi(\nu) g
\rangle \langle \pi(\nu) h, \pi(\chi)^* f_2\rangle \, d\nu 
\]
is continuous, constant on cosets of $\Delta$ (i.e., $\Delta$-periodic), and $\varphi(\chi)
\le B_g^{1/2} B_h^{1/2} \Vert f_1 \Vert \Vert f_2\Vert$ for all $\chi \in G\times \ghat$.  Furthermore,
the generators $g$ and $h$ satisfy: 
$\vert \langle h,g\rangle \vert \le \vol{\Delta} B_g^{1/2} B_h^{1/2}$.
\end{lemma}
\begin{proof}
By the Cauchy-Schwarz inequality, we see that
\begin{multline}
\label{eq:varphi-bounded}
 \int_{\Delta} \big\vert\langle \pi(\chi)^* f_1, \pi(\nu) g \rangle \langle \pi(\nu) h, \pi(\chi)^* f_2\rangle\big\vert \, d\nu \, d\dot\chi \\
\le \Big( \int_{\Delta} \vert\langle \pi(\chi)^* f_1, \pi(\nu) g \rangle\vert^2 \, d\nu \Big)^{1/2}
\Big( \int_{\Delta} \vert\langle \pi(\chi)^* f_2, \pi(\nu) h \rangle\vert^2 \, d\nu \Big)^{1/2}  \\
\le  B_g^{1/2} B_h^{1/2} \Vert f_1\Vert \, \Vert f_2 \Vert.
\end{multline}
This computation shows that $\varphi$ is well-defined and bounded.  The continuity of $\varphi$
can be shown using the Bessel property of $\gaborG{g}$ and $\gaborG{h}$ and the strong continuity of
$\nu \mapsto \pi(\nu)$. 
 The fact that the mapping $\varphi$ is
$\Delta$-periodic is easily verified. We have only left to prove the furthermore-part. By
Lemma~\ref{le:stft-norm-preserving} the mapping $\chi \mapsto \langle f_1, \pi(\chi) g\rangle \langle
\pi(\chi) h, f_2\rangle$ lies in $L^1(G\times\ghat)$. We can therefore apply Weil's formula for the
subgroup $\Delta$ to find that
\begin{equation} \label{eq:frame-implies-cocompact-1} \int_{G\times\ghat} \innerprod{f_1}{\pi(\chi)
    g} \innerprod{\pi(\chi) h}{f_2}\, d\chi = \int_{(G\times \ghat)/\Delta} \int_{\Delta} \innerprod{
    f_1}{\pi(\chi\nu) g} \innerprod{\pi(\chi \nu) h}{f_2} \, d\nu \, d\dot{\chi}.
\end{equation}
For any $\chi=(x,\omega)\in G\times \ghat$ and $\nu = (\lambda,\gamma)\in G\times\ghat$ we have, by
\eqref{eq:cr3},
\[ 
\innerprod{f_1}{\pi(\chi\nu) g} \innerprod{\pi(\chi\nu) h}{f_2} = \innerprod{\pi(\chi)^*
  f_1}{\pi(\nu) g} \innerprod{\pi(\nu) h}{\pi(\chi)^* f_2}. 
\]
With this, equation \eqref{eq:frame-implies-cocompact-1} becomes
\begin{equation} 
\label{eq:frame-implies-cocompact-2}
 \int_{G\times\ghat} \langle f_1, \pi(\chi) g\rangle \langle \pi(\chi) h, f_2\rangle \, d\chi  =
 \int_{(G\times \ghat)/\Delta} \int_{\Delta} \langle \pi(\chi)^* f_1, \pi(\nu) g \rangle \langle
 \pi(\nu) h, \pi(\chi)^* f_2\rangle \, d\nu \, d\dot\chi. 
\end{equation}
Lemma~\ref{le:stft-norm-preserving}, together with \eqref{eq:varphi-bounded} and \eqref{eq:frame-implies-cocompact-2}, yields
\begin{align*}
\vert \langle f_1,f_2\rangle\langle h,g\rangle \vert & = \bigg\vert \int_{G\times\ghat} \langle f_1, \pi(\nu) g\rangle \langle \pi(\nu) h, f_2\rangle \, d\nu\bigg\vert\\
& \le \int_{(G\times \ghat)/\Delta} \int_{\Delta} \big\vert\langle \pi(\chi)^* f_1, \pi(\nu) g \rangle \langle \pi(\nu) h, \pi(\chi)^* f_2\rangle\big\vert \, d\nu \, d\dot\chi \\
& \le \int_{(G\times\ghat)/\Delta} B_g^{1/2} B_h^{1/2} \Vert f_1\Vert \, \Vert f_2 \Vert \, d\dot\chi.
\end{align*}
The bound on $\abs{\innerprod{h}{g}}$ now follows from taking $f_1=f_2$. 
\end{proof}

\section{Density results}
\label{sec:density}


Our first density result shows that co-compactness of $\Delta\subset G\times \ghat$ is a necessary condition
for the frame property of a Gabor system $\gaborG{g}$. 
\begin{theorem} 
\label{th:frame-implies-cocompact} 
Let $\Delta$ be a closed subgroup of $G\times\ghat$, and let $g \in L^2(G)$. If $\gaborG{g}$ is a frame for $L^2(G)$ with bounds $0<A\le B<\infty$, then the following holds:
\begin{enumerate}[(i)]
\item the quotient group $(G\times \ghat)/\Delta$ is compact, i.e., $\vol{\Delta} < \infty$,
\item $A \vol{\Delta} \le \Vert g\Vert^2 \le B \vol{\Delta}.$
\end{enumerate}  
\end{theorem}
\begin{proof}
By Lemma~\ref{le:stft-norm-preserving} the mapping $\chi \mapsto \vert\langle f, \pi(\chi) g\rangle
\vert^2$ lies in $L^1(G\times\ghat)$. Weil's formula for the subgroup $\Delta$ then gives 
 \begin{align} \label{eq:s-frame-implies-cocompact-1}
 \int_{G\times\ghat} \abs{ \innerprod{f}{\pi(\chi) g}}^2  d\chi  &= \int_{(G\times \ghat)/\Delta}
 \int_{\Delta} \abs{\innerprod{f}{\pi(\chi\nu) g}}^2  d\nu \, d\dot{\chi} \nonumber \\ &= \int_{(G\times \ghat)/\Delta}
 \int_{\Delta} \abs{\innerprod{\pi(\chi)^* f}{\pi(\nu) g}}^2  d\nu \, d\dot{\chi},
 \end{align}
 where $\vert \langle f, \pi(\chi\nu) g\rangle \vert = \vert \langle \pi(\chi)^* f, \pi(\nu) g
 \rangle\vert$ follows from~\eqref{eq:cr3}.  The frame assumption of $\{\pi(\nu) g\}_{\nu\in
   \Delta}$ states that
\[ A \, \Vert f \Vert^2 \le \int_{\Delta} \vert \langle  f, \pi(\nu) g \rangle \vert^2 \, d\nu \le B
\, \Vert f \Vert^2 \quad \text{for all } f\in L^2(G).\]
Integrating the lower frame inequality for $\pi(\chi)^* f$ over $(G\times \ghat)/\Delta$  yields the following: 
 \begin{equation*} 
  A \norm{f}^2 \int_{(G\times \ghat)/\Delta}d\dot{\chi} = A  \int_{(G\times \ghat)/\Delta} \norm{\pi(\chi)^* f}^2 d\dot{\chi} \le  \int_{(G\times \ghat)/\Delta} \int_{\Delta} \vert \langle \pi(\chi)^* f, \pi(\nu) g \rangle \vert^2 \, d\nu \, d\dot{\chi}.
 \end{equation*}
By Lemma~\ref{le:stft-norm-preserving}(i) and \eqref{eq:s-frame-implies-cocompact-1} the term on the far right equals $\norm{f}^2 \norm{g}^2$. We conclude that
\[ A \int_{(G\times\ghat)/\Delta} d\dot\chi \le \Vert g\Vert^2  <  \infty.\] 
The measure of the quotient group $(G\times \ghat)/\Delta$ is finite if, and only if $(G\times
\ghat)/\Delta$ is compact. This proves (i) and the lower inequality in (ii). To get the upper bound
in (ii), we look at the upper frame inequality and proceed as above to find: 
\[ 
\Vert g \Vert^2 \, \Vert f \Vert^2 = \int_{G\times\ghat}  \vert \langle f, \pi(\nu) g \rangle
\vert^2 \, d\nu \le B \,\Vert f \Vert^2 \int_{(G\times\ghat)/\Delta} d\dot\chi.
\]
\end{proof}


The assertions of Theorem~\ref{th:frame-implies-cocompact} are
not true in general if we assume that $\gaborG{g}$ is a basic frame,
i.e., a frame for its closed linear span, instead of assuming that $\gaborG{g}$
is a total frame. 
Density results for basic frames in the case of lattice Gabor
systems in $L^2(\R^n)$ have recently been obtained in
\cite{MR3073252}; we will not consider such extensions here.

  Let us consider some implications of the density result in
  Theorem~\ref{th:frame-implies-cocompact} for a couple of specific locally compact abelian
  groups. The first result shows an extreme behavior for the $p$-adic numbers. From
  Lemma~\ref{le:stft-norm-preserving} we know that for the  short-time Fourier transform any nonzero window
  will generate a Gabor frames. However, for the $p$-adic numbers no other time-frequency subgroup will have a frame generator.   
  \begin{corollary} 
    For a prime number $p$, consider the $p$-adic numbers $\mathbb Q_p$. Let $\Delta$ be a closed
    subgroup of $\mathbb Q_p \times \widehat{\mathbb Q}_p$. If $\gaborG{g}$ is a frame for some
    $g\in L^2(\mathbb Q_p)$, then $\Delta = \mathbb Q_p \times \widehat{\mathbb Q}_p$.
\end{corollary}
\begin{proof} 
  The result follows from Theorem~\ref{th:frame-implies-cocompact} together with the fact that the
  only co-compact subgroup of $\mathbb Q_p \times \widehat{\mathbb Q}_p$ is the entire group itself.
\end{proof}

\begin{corollary} 
\label{col:density-for-Rn}
  Let $g\in L^2(\R^n)$. If the system $\gaborG{g}$ is a regular Gabor frame for $L^2(\R^n)$, then the
  closed subgroup $\Delta$ is of the form $A (\Z^k\times\R^{2n-k})$ for some $A\in
  \mathrm{GL}_{\R}(2n)$ and $0\le k \le 2n$.
\end{corollary}  
\begin{proof}
  Any closed subgroup $\Delta$ of $\R^{2n}$ is isomorphic to $\{0\}^\ell \times \Z^k \times
  \R^{2n-k-\ell}$ for $0 \le k + \ell \le 2n$. The subgroup $\{0\}^\ell \times \Z^k \times
  \R^{2n-k-\ell}$ is co-compact exactly when $\ell=0$. Hence, by
  Theorem~\ref{th:frame-implies-cocompact}, the subgroup $\Delta$ is of the form $A
  (\Z^k\times\R^{2n-k})$ for some $A\in \mathrm{GL}_{\R}(2n)$ and $0\le k \le 2n$.
\end{proof}

The next results relate the norm of a Gabor frame generator to the
subgroup size $\vol{\Delta}$.

\begin{corollary} 
\label{col:WR-tight}
Let $\Delta$ be a closed subgroup of $G\times\ghat$, and let $g \in L^2(G)$.  If $\gaborG{g}$ is a
tight frame with bound $A$, then $\gaborG[\Delta^{\circ}]{g}$ is an orthogonal system with
$\norm{g}^2 = \vol{\Delta} A$.
\end{corollary}
\begin{proof}
  The canonical dual frame of $\gaborG{g}$ is $\gaborG{\frac1A g}$. From
  Theorem~\ref{th:frame-implies-cocompact} we know that $\Delta$ is co-compact and $\Vert g \Vert^2 = \vol{\Delta}$. By
  Lemma~\ref{th:biorth-identity} it follows that $\gaborG{g}$ is an orthogonal system with 
  $\innerprod{\frac1A g}{g}=\vol{\Delta}$. 
\end{proof}

\begin{corollary}
\label{col:WR-Parseval}
Let $\Delta$ be a closed subgroup of $G\times\ghat$, and let $g \in L^2(G)$.  If $\gaborG{g}$ is a
frame, then $\norm{S^{-1/2}g}^2=\vol{\Delta}$.
\end{corollary}
\begin{proof}
  Lemma~\ref{le:wex-raz-proof-inner-product-rewrite} and
  Remark~\ref{re:canonical-gabor-frames} show that $\gaborG{S^{-1/2}g}$ is a Parseval frame. The result now follows from Corollary~\ref{col:WR-tight}.
\end{proof}

If we in addition to co-compactness in Theorem~\ref{th:frame-implies-cocompact} assume that $\Delta$ is discrete,
i.e., a uniform lattice, we have a quantitative density theorem.
\begin{theorem}
\label{thm:density-frame-uniform} Let $\Delta$ be a discrete subgroup
of $G\times\ghat$ equipped with the counting measure, and let $g \in L^2(G)$. If $\gaborG{g}$ is a frame
for $L^2(G)$, then $\Delta$ is a uniform lattice with $\vol{\Delta}\le 1$.
\end{theorem}
\begin{proof} By Theorem~\ref{th:frame-implies-cocompact} it follows
  that $\Delta$ is a uniform lattice.
From Corollary~\ref{col:WR-Parseval} we have that $\Vert
  S^{-1/2}g \Vert^2 =\vol{\Delta}$. Taking $f=S^{-1/2}g$ in the upper
  frame inequality for $\gaborG{S^{-1/2}g}$ yields, using that
  $\Delta$ is discrete, that $\Vert S^{-1/2} g \Vert^2 \le 1$. We
  conclude that $\vol{\Delta}\le 1.$
\end{proof}

In Theorem~\ref{thm:density-frame-uniform} the assumption that $\Delta$ is discrete is essential for
the bound $\vol{\Delta}\le 1$.  Indeed, in Example~\ref{exa:construction-of-frames-in-L2Rn} in
Section~\ref{thm:duality-principle} we will show that in $L^2(\R^n)$ for $\Delta \subset \R^{2n}$
separable and co-compact, but non-discrete, it will always be possible to construct a frame $\gaborG{g}$
regardless of the value of $\vol{\Delta}$. The construction relies on the duality
principle, which is why the example is relegated to Section~\ref{thm:duality-principle}.


\begin{theorem}
\label{thm:riesz-density}
Let $\Delta$ be a closed subgroup of $G\times\ghat$, and let $g\in
L^2(G)$. Then $\gaborG{g}$ is a total Riesz family for $L^2(G)$
if, and only if, $\Delta$ is a uniform lattice, $\vol{\Delta} = 1$,
and $\gaborG{g}$ is a frame.
\end{theorem}
\begin{proof} 
  Assume that $\gaborG{g}$ is a total Riesz family. By Proposition
 ~\ref{th:rieszseq-implies-discrete} and Theorem~\ref{th:frame-implies-cocompact}, the subgroup $\Delta$ is discrete and co-compact. Hence,
  $\gaborG{g}$ is a Riesz basis, and $\gaborG{S^{-1/2}g}$ is therefore an orthonormal basis for
  $L^2(G)$, see Remark~\ref{re:canonical-gabor-frames}. It follows that $\Vert S^{-1/2}g\Vert^2 =
  1$. Furthermore, by Corollary~\ref{col:WR-Parseval}, we have that $\Vert S^{-1/2}g\Vert^2 =
  \vol{\Delta}$. Hence $\vol{\Delta} = 1$.

  For the converse implication note that $\Vert S^{-1/2} g \Vert ^2 = \vol{\Delta} = 1$ by
  Corollary~\ref{col:WR-Parseval}. By isometry of the time-frequency shifts we see that
  $\innerprod{\pi(\nu)g}{\pi(\nu)S^{-1}g}=1$ for all $\nu \in \Delta$. By Theorem 5.4.7~and
  Proposition~5.4.8 in \cite{MR1946982}, it follows that $\gaborG{g}$ and $\gaborG{S^{-1}g}$ are
  dual Riesz bases, and we conclude that $\gaborG{g}$ is a Riesz basis. Alternatively, we can arrive
  at this conclusion as follows. Again by isometry of $\pi(\nu)$, we see that $\Vert \pi(\nu)
  S^{-1/2} g \Vert ^2 = 1$ for all $\nu\in\Delta$. Hence
  $\gaborG{S^{-1/2}g}$ is a discrete Parseval frame
  whose elements have norm $1$, and thus it is actually an orthonormal basis. As $\gaborG{g}$ is the
  image of the orthonormal basis $\gaborG{S^{-1/2}g}$ under the bounded, invertible operator
  $S^{1/2}$, it follows that $\gaborG{g}$ is a Riesz basis for $L^2(G)$. Here, we tacitly used
  Lemma~\ref{thm:frame-op-commutes-time-freq}.
\end{proof}

Owing to Theorem~\ref{thm:density-frame-uniform} discrete Gabor frames $\gaborG{g}$, for which
$\vol{\Delta}=1$, are called \emph{critically sampled}.  Let us for a moment consider critically
sampled \emph{separable} Gabor systems that are systems of the form $\gaborG[\Lambda \times
\Lambda^\perp]{g} = \{E_{\gamma}T_{\lambda}g\}_{\lambda\in\Lambda,\gamma\in\Lambda^{\perp}}$ for
some closed subgroup $\Lambda$ of $G$. The following density result is a slightly stronger variant
of Theorem~\ref{thm:riesz-density} for the special case of separable critical sampling.

\begin{corollary} 
\label{cor:crit-sampled-implies-uniform-lattice}
Let $\LL$ be a closed subgroup of $G$, and let $g \in L^2(G)$. If
$\{E_{\gamma}T_{\lambda}g\}_{\lambda\in\Lambda,\gamma\in\Lambda^{\perp}}$ is a frame for $L^2(G)$,
then $\Lambda$ is a uniform lattice of $G$ and
$\{E_{\gamma}T_{\lambda}g\}_{\lambda\in\Lambda,\gamma\in\Lambda^{\perp}}$ is a Riesz basis.
\end{corollary}
\begin{proof} 
  By Theorem~\ref{th:frame-implies-cocompact} the quotient group $(G\times
  \ghat)/(\Lambda\times\Lambda^{\perp})\cong(G/\Lambda \times \ghat/\Lambda^{\perp})$ has to be
  compact. This only happens if both $G/\Lambda$ and $\ghat/\Lambda^{\perp}$ are compact. Hence
  $\Lambda$ is a co-compact subgroup of $G$ and $\ghat/\Lambda^{\perp}\cong \widehat\Lambda$ is
  compact in $\ghat$. The latter conclusion implies that $\Lambda$ is discrete. Thus $\Lambda$ is a
  uniform lattice. The fact that
  $\{E_{\gamma}T_{\lambda}g\}_{\lambda\in\Lambda,\gamma\in\Lambda^{\perp}}$ is a Riesz basis now
  follows from Theorem~\ref{thm:riesz-density}. 
\end{proof}

From Corollary~\ref{cor:crit-sampled-implies-uniform-lattice} we see that if there are no uniform
lattices in $G$, then there do not exist any separable, critically sampled Gabor frames for
$L^2(G)$.  For the Pr\"ufer $p$-group $G=\Z(p^\infty)$ the only uniform lattice is the Pr\"ufer
$p$-group itself, therefore there is only one type of critically sampled Gabor system, namely
$\{T_{\lambda}g\}_{\lambda\in \Z(p^\infty)}$.

Corollary~\ref{cor:crit-sampled-implies-uniform-lattice} has the following direct implications.
\begin{corollary} 
\label{cor:TI-frame-imply-discrete}
Let $g\in L^2(G)$. \begin{enumerate}[(i)]
\item If $\{T_{\lambda}g\}_{\lambda\in G}$ is a frame for $L^2(G)$, then $G$ is discrete. 
\item If $\{E_{\gamma}g\}_{\gamma\in \ghat}$ is a frame for $L^2(G)$, then $G$ is compact.
\end{enumerate}
\end{corollary}

Let us end this section with commenting on yet another difference between discrete and non-discrete
Gabor systems. For a full-rank lattice $\Delta$ in $\R^{2n}$, Bekka~\cite{MR2078261} proved (using
von Neumann algebra techniques) that there exists $g\in L^2(\R^n)$ so that $\gaborG{g}$ is a frame
if, and only if, there exists $g\in L^2(\R^n)$ so that $\gaborG{g}$ is total, i.e., the linear span of
the functions in $\gaborG{g}$ is dense in $L^2(\R^n)$. This equivalence is not true for non-discrete
Gabor systems, e.g., take $\Delta=\R^{n} \times \{0\}^n$. Then $\gabor{g}=\set{T_\lambda g}_{\lambda
  \in \R^n}$, and it follows from Corollary~\ref{cor:TI-frame-imply-discrete} that no $g \in
L^2(\R)$ can generate a Gabor frame since $\R^n$ is not
discrete, see also \cite{MR1721808}. 
However, for any function $g$ such that $\hat{g}(\omega)\neq 0$ for a.e.\ $\omega \in \widehat{\R}^n$,
we see that
\[ 
0=\innerprod{T_\lambda g}{f} = \innerprods{E_{-\lambda} \hat{g}}{\hat{f}} =
\cF^{-1}(\hat{g}\bar{\hat{f}})(-\lambda) \quad
\text{for all } \lambda \in \R
\]
implies that $f=0$, hence $\set{T_\lambda g}_{\lambda \in \R^n}$ is total.
This argument obviously also works for $\Delta=G \times 
\{0\} \subset G \times \ghat$. In general, co-compactness of $\Delta \subset G \times \ghat$ is not necessary for $\gaborG{g}$ to
be total.  


%

\section{Duality results}
\label{sec:duality}

To simplify the formulation of the duality result and to avoid working
with infinite subgroup sizes, we introduce the following variant of
$\vol{\Delta}$:
\[ 
 \volto{\Delta} = \begin{cases} \vol{\Delta}=\int_{(G\times\ghat)/\Delta} 1 \,
     d\dot{\chi} & \text{if $(G\times\ghat)/\Delta$ is compact,} \\ 1 &
     \text{otherwise.}\end{cases}  
\]
The precise value of  $\volto{\Delta}$ for non-co-compact subgroups
$\Delta$ is not important as we just need that $\volto{\Delta}$ is
finite for all closed subgroups. 

\subsection{The Wexler-Raz biorthogonality relations}
\begin{theorem} \label{th:wex-raz-nonsep} Let $\Delta$ be a closed subgroup of
  $G\times\ghat$, and let $g,h\in L^2(G)$. Suppose that $\gaborG{g}$ and  $\gaborG{h}$ are Bessel systems. Then the following statements are equivalent:
\begin{enumerate}[(i)]
\item $\gaborG{g}$ and  $\gaborG{h}$ are dual frames,
\item $\langle h,\pi(\mu)g\rangle = \volto{\Delta} \delta_{\mu,e} \ \, \text{for all } \mu\in \Delta^{\circ}$.
\end{enumerate} 
If either and hence both of the assertions hold, then $(G\times\ghat)/\Delta$ is compact. 
\end{theorem}

\begin{proof}
Assume that $\gaborG{g}$ and  $\gaborG{h}$ are dual frames for $L^2(G)$. By Theorem~\ref{th:frame-implies-cocompact} this implies that $\Delta$ is co-compact. It follows by Lemma~\ref{th:biorth-identity} that 
\[ \langle h,\pi(\mu)g\rangle = \vol{\Delta} \, \delta_{\mu,e} \quad \text{for all } \mu\in \Delta^{\circ}.\]

Assume now that $\langle h, \pi(\mu)g\rangle = \volto{\Delta} \delta_{\mu,e}$ for $\mu\in
\Delta^{\circ}$. Suppose $\Delta$ is not co-compact. Then the cardinality of $\Delta^{\circ}$ is
uncountable. However, this contradicts the assumption that $h$ and $\pi(\mu)g$ are biorthogonal for each $\mu\in
\Delta^{\circ}$ since $L^2(G)$ is separable. Thus, the subgroup $\Delta$ is co-compact.
By Lemma~\ref{le:wex-raz-proof-inner-product-rewrite} we have
\begin{equation}
\label{eq:wex-raz-fourier-series-coeff} \vol{\Delta} \, \delta_{\mu,e} \langle \pi(\mu) f_{1},f_{2}\rangle =
\int_{(G\times \ghat)/\Delta} \overline{\omega(\alpha)} \beta(x)  \varphi(\chi)  d\dot\chi,
\end{equation}
where $\varphi(\chi)=\int_{\Delta} \langle \pi(\chi)^* f_{1}, \pi(\nu)g\rangle\langle \pi(\nu) h,
\pi(\chi)^* f_{2}\rangle \, d\nu$ and $f_{1},f_{2} \in L^2(G)$ are arbitrary. 
By Lemma~\ref{pr:gabor-fundamental-function-bounded} the mapping $\varphi$
 is a bounded function on the 
compact domain $(G\times \ghat)/\Delta$. It therefore has a Fourier 
series indexed by the dual group of $(G\times\ghat)/\Delta$, which is 
topologically isomorphic to the discrete group $\Delta^{\circ}$. The right hand side of
equation~\eqref{eq:wex-raz-fourier-series-coeff} are the Fourier coefficients of $\varphi$. 
Indeed, by assumption, all but one are identically zero. We thus have 
the Fourier series expansion
\begin{equation} \label{eq:wex-raz-fourier-series-almost-all}
\varphi(\chi)= \vol{\Delta}^{-1} \sum_{\mu\in\Delta^{\circ}}  \vol{\Delta} \, \delta_{\mu,e} \langle \pi(\mu) f_{1},f_{2}\rangle = \langle f_{1}, f_{2}\rangle,
\end{equation}
which holds for almost all $\chi\in G\times\ghat$.
Moreover, by Lemma~\ref{pr:gabor-fundamental-function-bounded} the function $\varphi$ is
continuous. Hence, equality \eqref{eq:wex-raz-fourier-series-almost-all} holds pointwise; in particular, for
$\chi=e$, it yields 
\begin{equation*} 
\varphi(e)=\int_{\Delta} \langle f_{1}, \pi(\nu)g\rangle\langle \pi(\nu) h, f_{2}\rangle \, d\nu = \langle f_{1}, f_{2}\rangle.
\end{equation*}
Thus $\{\pi(\nu)g\}_{\nu\in\Delta}$ and $\{\pi(\nu)h\}_{\nu\in\Delta}$ are dual frames for $L^2(G)$. 
\end{proof}

\subsection{The Janssen representation}

For any closed subgroup $\Delta$ in $G\times\ghat$ Lemma~\ref{le:wex-raz-proof-inner-product-rewrite} states that the Fourier transform of
the $\Delta$-periodic function 
\[ \varphi: G\times\ghat \to \C , \ \chi \mapsto \int_{\Delta} \langle \pi(\chi)^* f_1, \pi(\nu)g\rangle\langle \pi(\nu) h, \pi(\chi)^* f_2\rangle \, d\nu\]
is given by
\[ \hat\varphi(\mu) = \langle h, \pi(\mu) g\rangle \langle \pi(\mu) f_1,f_2\rangle , \ \ \mu\in\Delta^{\circ}.\]
Indeed, $\varphi\in L^1((G\times\ghat)/\Delta)$
since
\[ \int_{(G\times\ghat)/\Delta} \Big\vert \int_{\Delta} \langle \pi(\chi)^* f_1, \pi(\nu)g\rangle\langle \pi(\nu) h, \pi(\chi)^* f_2\rangle \, d\nu \, \Big\vert \, d\dot\chi \le \Vert f_1 \Vert \,\Vert f_2 \Vert \,\Vert g \Vert \,\Vert h \Vert .\]
Using the Fourier inversion formula, we then recover the \emph{fundamental
  identity in Gabor analysis} \eqref{eq:Janssen-rep-function-eq2} for Gabor systems in $L^2(G)$ by Rieffel~\cite{MR941652}. 
\begin{theorem} \label{th:janssen-fourier-inversion} 
Let $\Delta$ be any closed subgroup of $G\times\ghat$, and define
$\varphi\in L^1((G\times\ghat)/\Delta)$ as above.
Assume that $\hat\varphi \in L^1(\Delta^{\circ})$. For each
$\mu=(\alpha,\beta)\in \Delta^{\circ}$ we have:
\begin{equation} \label{eq:Janssen-rep-function} 
\int_{\Delta} \langle \pi(\chi)^* f_1, \pi(\nu)g\rangle\langle 
\pi(\nu) h, \pi(\chi)^* f_2\rangle \, d\nu = 
\int_{\Delta^{\circ}} \omega(\alpha) \overline{\beta(x)} \langle h, 
\pi(\mu) g\rangle \langle \pi(\mu) f_1,f_2\rangle \, d\mu
\end{equation}
for almost every $\chi = (x,\omega)\in G\times\ghat$.
If, furthermore, $\varphi$ is continuous, then the inversion
formula~\eqref{eq:Janssen-rep-function} holds pointwise; for
$\mu=\chi=(e_G,e_{\ghat})$ we find:
\begin{equation} \label{eq:Janssen-rep-function-eq2} 
\int_{\Delta} \langle f_1, \pi(\nu)g\rangle\langle 
\pi(\nu) h, f_2\rangle \, d\nu = 
\int_{\Delta^{\circ}} \langle h, 
\pi(\mu) g\rangle \langle \pi(\mu) f_1,f_2\rangle \, d\mu.
\end{equation}
\end{theorem}

\begin{corollary} \label{cor:janssen-rep}
Let $\Delta$ be a closed subgroup of $G\times \ghat$, and let $g,h\in L^2(G)$. Suppose
that $\gaborG{g}$ and $\gaborG{h}$ are Bessel systems. If the functions
$g,h$ and $f_1,f_2\in L^2(G)$ satisfy 
\begin{equation} \label{eq:cond-A} 
\int_{\Delta^{\circ}} \vert \langle h, \pi(\mu)g\rangle \langle \pi(\mu)f_1,f_2\rangle \vert \, d\mu <
\infty,
\end{equation}
then
\begin{equation} 
\label{eq:Janssen-rep} 
\langle S_{g,h} f_1,f_2\rangle = \int_{\Delta} \langle f_1, \pi(\nu)g\rangle\langle 
\pi(\nu) h, f_2\rangle \, d\nu = 
\int_{\Delta^{\circ}} \langle h, 
\pi(\mu) g\rangle \langle \pi(\mu) f_1,f_2\rangle \, d\mu.
\end{equation}
\end{corollary}
\begin{proof} 
  The Bessel assumption, Lemma~\ref{pr:gabor-fundamental-function-bounded} and \eqref{eq:cond-A}
  ensure that the conclusion \eqref{eq:Janssen-rep-function-eq2} of Theorem~\ref{th:janssen-fourier-inversion} holds.
\end{proof}

In \cite{MR2264211} the authors determine sufficient conditions on the functions $g,h,f_{1},f_{2}$
under which~\eqref{eq:Janssen-rep-function-eq2} holds. In particular, we mention
that~\eqref{eq:Janssen-rep-function-eq2} holds if $g,h\in L^2(G)$ and $f_1,f_2$ belong to the
Feichtinger algebra $S_0(G)$, cf.\ \cite{MR2264211} and Theorem~\ref{thm:frame-op-is-in-s0-for-g-in-s0} and
Corollary~\ref{cor:Bessel-system-with-g-in-s0} in the appendix.

Assume that $\Delta$ is a closed, co-compact subgroup of $G \times \ghat$. The measure
on $\Delta^{\circ}$ in the right hand side of \eqref{eq:Janssen-rep} is then given by $\vol{\Delta}^{-1}
\sum_{\mu\in\Delta^{\circ}}$. The pair $(g,\Delta)$ satisfies \emph{condition A} if
$\sum_{\mu\in\Delta^{\circ}} \vert \langle g, \pi(\mu) g\rangle \vert < \infty$. Now, if the Gabor
system $\{\pi(\nu)g\}_{\nu\in\Delta}$ is a Bessel family \emph{and} condition A holds, then
\eqref{eq:Janssen-rep} yields the \emph{Janssen representation} of the frame operator:
\[ 
S_{g} = \vol{\Delta}^{-1} \sum_{\mu\in\Delta^{\circ}} \langle g, \pi(\mu)g\rangle \pi(\mu)
\]
with absolute convergence in the (uniform) operator norm. 
 It follows from Proposition~\ref{prop:Vfg-in-s0}
 and the comments preceding Corollary \ref{cor:Bessel-system-with-g-in-s0} that any $g\in
S_{0}(G)$ satisfies condition A. 
The mixed frame operator $S_{g,h}$, $g,h\in L^2(G)$, has a similar Janssen representation. 

\subsection{The duality principle}

In this section we proof an extended version of the duality theorem for Gabor frames. The original
result on separable lattice Gabor systems on $L^2(\R^d)$ goes back to Daubechies, Landau and
Landau~\cite{MR1350701}, Janssen~\cite{MR1350700}, and Ron and Shen~\cite{MR1460623}. Our proof is
inspired by one direction of Janssen's proof; the important fact is that Janssen's computations
carry over from the setting of discrete, separable Gabor systems in $L^2(\R)$ to regular,
non-separable Gabor systems in $L^2(G)$. From this idea, we prove that for any closed
subgroup $\Delta$ in $G\times\ghat$ the Gabor system $\{\pi(\nu)g\}_{\nu\in\Delta}$ is a Bessel
system with bound $B$, if, and only if, $\{\pi(\mu)g\}_{\mu\in\Delta^{\circ}}$ is a Bessel system
with bound $B$. We remind the reader that a Gabor system is a Bessel system with bound $B$
\emph{with respect to} the measure on the associated time-frequency subgroup.  In case $\Delta$ is
co-compact, the measure on $\Delta^{\circ}$ is $\vol{\Delta}^{-1} \sum_{\mu\in\Delta^{\circ}}$, see
Remark~\ref{rem:group-size}, and the Bessel duality principle in Theorem~\ref{thm:Bessel-duality}
states that for $g\in L^2(G)$ and $B>0$ we have:
\[ 
\int_{\Delta} \vert \langle f, \pi(\nu)g\rangle \vert^2 \, d\nu \le B \,
\Vert f \Vert^2  \quad \text{if, and only if} \quad
\sum_{\mu\in\Delta^{\circ}} \vert \langle f, \pi(\mu)g\rangle \vert^2 \le
\vol{\Delta} B \, \Vert f \Vert^2 
\]
for all $f\in L^2(G)$.
When $\Delta$ is a full-rank lattice in $\R^{2n}$, the Bessel duality result is well-known, and the
result is stated in \cite[Proposition 3.5.10]{MR1601107}, albeit without bounds.
Note, however, that the Bessel duality principle is true for any
closed subgroup of $G \times \ghat$ and that neither co-compactness nor
discreteness is needed.
The generalized duality principle will then follow from the Bessel
duality principle and the Wexler-Raz biorthogonality relations using general frame theory. 

\begin{theorem}
\label{thm:Bessel-duality}
Let $\Delta$ be a closed subgroup of $G\times\ghat$, and let $g \in
L^2(G)$. Then $\gaborG{g}$ is a Bessel system with bound $B$ if, and only if, $\gaborG[\Delta^{\circ}]{g}$ is a Bessel system
with bound $B$.
\end{theorem}

\begin{proof} 
By symmetry of the Bessel duality principle, we only have to prove one of the two
implications. We assume that $g \in L^2(G)$ and that $\{\pi(\nu)g\}_{\nu\in\Delta}$ is a
Bessel system with bound $B$. 
In other words, we assume that the analysis operator $C_{g,\Delta}: L^2(G) \to
L^2(\Delta)$ is bounded with operator bound $\sqrt B$. Therefore, its
adjoint, the synthesis operator $D_{g,\Delta}$ given weakly by
\[
 \innerprod{D_{g,\Delta} a}{h}  = \int_{\Delta} a(\nu) \langle \pi(\nu) g, h \rangle \,
 d\nu  \qquad \text{for all } h\in L^2(G)
\]
is also bounded by $\sqrt B\,$:
\begin{equation}
\label{eq:bessel-synth-assumption}
  \Vert D_{g,\Delta} \varphi \Vert_{L^2(G)}^2 \le B \, \Vert \varphi
  \Vert_{L^2(\Delta)}^2 \quad \text{for all } \varphi \in L^2(\Delta).
\end{equation}
For each $\chi\in G \times\widehat{G}$, we define $\varphi(\nu) =
\langle \pi(\chi) h, \pi(\nu) f \rangle,\, \nu\in\Delta$, where
$f,h\in S_0(G)$, normalized so that $\Vert h \Vert_{L^2(G)} = 1$. 
It follows from Proposition~\ref{prop:Vfg-in-s0} together with the comments preceding Corollary \ref{cor:Bessel-system-with-g-in-s0} that $\varphi \in S_{0}(\Delta)$.

By a change of variables and using the properties of the time-frequency shift operator we find that
\begin{align*}
 \Vert D_{g,\Delta} \varphi \Vert_{L^2(G)}^2 =& \innerprod{D_{g,\Delta} \varphi}{D_{g,\Delta} \varphi}  
 = \int_{\Delta} \varphi(\nu) \int_{\Delta} \langle \pi(\nu) g, \pi(\nu')g\rangle \, \overline{\varphi(\nu')} \, d\nu' \, d\nu \\
& = \int_{\Delta} \varphi(\nu) \int_{\Delta} \langle \pi(\nu) g, \pi(\nu')g\rangle \langle \pi(\nu') f, \pi(\chi) h\rangle \, d\nu' \, d\nu \\
& = \int_{\Delta} \varphi(\nu) \int_{\Delta} \langle \pi(\nu) g, \pi(\nu\nu')g\rangle \langle \pi(\nu\nu') f, \pi(\chi) h\rangle \, d\nu' \, d\nu \\
& = \int_{\Delta} \varphi(\nu) \int_{\Delta} \langle \pi(\nu) g, \pi(\nu)\pi(\nu')g\rangle  \langle \pi(\nu)\pi(\nu') f, \pi(\chi) h\rangle \, d\nu' \, d\nu \\
& = \int_{\Delta} \varphi(\nu) \int_{\Delta} \langle g, \pi(\nu')g\rangle \langle \pi(\nu') f, \pi(\nu)^* \pi(\chi) h\rangle \, d\nu' \, d\nu.
\end{align*}
Note that the order of integration can be interchanged by
Fubini's theorem since $\varphi\in S_{0}(\Delta)\subset L^1(\Delta)$.

Since $f$ and $h$ belong to $S_0(G)$ and $g\in L^2(G)$, the
fundamental equation in Gabor analysis
\eqref{eq:Janssen-rep-function-eq2} holds (see the comments following
Corollary~\ref{cor:janssen-rep}). Hence, 
\begin{align}
\label{eq:synthesis-op-after-FIGA}
 \Vert D_{g,\Delta} \varphi \Vert_{L^2(G)}^2 
 = \int_{\Delta} \varphi(\nu) \int_{\Delta^{\circ}} \langle f, \pi(\mu) g \rangle \langle \pi(\mu) g, \pi(\nu)^* \pi(\chi) h\rangle \, d\mu \, d\nu.  
\end{align}
For the adjoint system $\gabor[\Delta^{\circ}]{g}$ it follows from
Corollary~\ref{cor:Bessel-system-with-g-in-s0} that the frame operator
$S_{g,\Delta^{\circ}}$ is well-defined on the subspace $S_{0}(G)$ of
$L^2(G)$. Thus, since $f \in S_0(G)$,
we have by definition that
\[
\langle S_{g,\Delta^{\circ}} f,
\pi(\nu)^*\pi(\chi)h\rangle = \int_{\Delta^{\circ}} \langle f,
\pi(\mu) g \rangle \langle \pi(\mu) g, \pi(\nu)^* \pi(\chi) h\rangle
\, d\mu.
\]
Hence, we can continue~\eqref{eq:synthesis-op-after-FIGA}:
\begin{align*}
 \Vert D_{g,\Delta} \varphi \Vert_{L^2(G)}^2 & = \int_{\Delta} \langle S_{g,\Delta^{\circ}} f, \pi(\nu)^*\pi(\chi)h\rangle \langle \pi(\nu)^* \pi(\chi) h,  f \rangle \, d\nu\\
& = \int_{\Delta} \langle S_{g,\Delta^{\circ}} f, \pi(\nu\chi)h\rangle \langle \pi(\nu\chi) h,  f \rangle \, d\nu
\end{align*}
where we have also used \eqref{eq:cr2}, \eqref{eq:cr3}, and a change of variables from
$\nu^{-1}$ to $\nu$. 
Thus, with our choice of $\varphi$, the
inequality~\eqref{eq:bessel-synth-assumption} becomes
\[ 
\int_{\Delta} \langle S_{g,\Delta^{\circ}} f, \pi(\nu\chi)h\rangle
\langle \pi(\nu\chi) h,  f \rangle \, d\nu \le B \, \int_{\Delta}
\vert \langle \pi(\nu\chi) h,  f \rangle \vert^2 \, d\nu.
\]
Integrating over the quotient $(G\times\ghat)/\Delta$ yields that
\begin{align*}
 \int_{(G\times\ghat) / \Delta} \int_{\Delta} \langle
S_{g,\Delta^{\circ}} f, \pi(\nu\chi)h\rangle \langle \pi(\nu\chi) h,
f \rangle \, d\nu \, d\dot\chi & \le B \, \int_{(G\times\ghat) / \Delta} \int_{\Delta} \vert \langle \pi(\nu\chi) h,  f \rangle \vert^2 \, d\nu \, d\dot\chi, \\
\intertext{and further by Weil's formula that} 
 \int_{G\times\ghat} \langle S_{g,\Delta^{\circ}} f, \pi(\chi)h\rangle
 \langle \pi(\chi) h,  f \rangle \, d\chi & \le  B \, \int_{G\times\ghat} \vert \langle \pi(\chi) h,  f \rangle \vert^2 \, d\chi.
\end{align*}
Using the orthogonality relations of the short-time Fourier transform
in Lemma~\ref{le:stft-norm-preserving}(iii), we arrive at 
\[ 
 \int_{\Delta^{\circ}}
\vert \langle f, \pi(\mu) g\rangle \vert^2 \, d\mu = \langle S_{g,\Delta^{\circ}} f,  f \rangle \le  B \, \Vert f
\Vert^{2}
\]
for any $f\in S_0(G)$. Since $S_0(G)$ is dense in $L^2(G)$, 
we conclude that $\{\pi(\mu) g\}_{\mu\in\Delta^{\circ}}$ is a Bessel system with bound $B$. 
\end{proof}


The duality principle can now be proven using general frame theory; the proof strategy is similar to
the proof of the duality principle for separable, co-compact subgroups in \cite{JakobsenCocompact2014}.

\begin{theorem} \label{thm:duality-principle} Let $\Delta$ be a closed
  subgroup of $G\times\ghat$ and let $g\in L^2(G)$. Then the following
  statements are equivalent:
\begin{enumerate}[(i)]
\item $\gaborG{g}$ is a Gabor frame with bounds $A$ and $B$,
\item $\gaborG[\Delta^{\circ}]{g}$ is a basic Riesz family with bounds $\volto{\Delta} \, A$ and $\volto{\Delta} \, B$.
\end{enumerate}
If either and hence both of the assertions hold, then $(G\times\ghat)/ \Delta$ is compact. 
\end{theorem}
\begin{proof} 
Suppose either (i) or (ii) holds. Then by Theorem~\ref{th:frame-implies-cocompact} and Proposition~\ref{th:rieszseq-implies-discrete} the group $(G\times\ghat)/\Delta$ is compact, and equivalently $\Delta^{\circ}$ is discrete.
Assume now that (i) holds. Then $\gabor{g}$ has the canonical dual frame $\gabor{S^{-1}g}$ with frame bounds $B^{-1}$ and $A^{-1}$. 
Therefore in particular $\gabor{g}$ and $\gabor{S^{-1}g}$ are Bessel systems with bounds $B$ and $A^{-1}$, respectively. By Theorem~\ref{thm:Bessel-duality} then also $\{\pi(\mu)(\vol{\Delta})^{-1/2} g\}_{\mu\in\Delta^{\circ}}$ and $\{\pi(\mu) (\vol{\Delta})^{-1/2} S^{-1} g\}_{\mu\in\Delta^{\circ}}$ are Bessel systems with respect to the counting measure on $\Delta^{\circ}$ with bound $B$ and $ A^{-1}$, respectively. By Theorem~\ref{th:wex-raz-nonsep} we also have that the duality of the frames $\gabor{g}$ and $\gabor{S^{-1}g}$ imply that  $\{\pi(\mu) (\vol{\Delta})^{-1/2} g\}_{\mu\in\Delta^{\circ}}$ and $\{\pi(\mu)(\vol{\Delta})^{-1/2} S^{-1}g\}_{\mu\in\Delta^{\circ}}$ are bi-orthogonal. By Theorem~\ref{thm:B} it now follows that
$\{\pi(\mu)(\vol{\Delta})^{-1/2} g\}_{\mu\in\Delta^{\circ}}$ is a basic Riesz family with bounds $A$
and $B$. The converse implication is similar where Theorem~\ref{thm:A} instead of
Theorem~\ref{thm:B} is used.
\end{proof}

Let us comment on a difference between the Bessel duality and the
duality principle. We have proven both results for any closed
subgroup $\Delta$ of $G \times \ghat$. However, for non-co-compact
subgroups, the duality principle
is vacuously true, in the sense that Theorem~\ref{th:frame-implies-cocompact} and Proposition~\ref{th:rieszseq-implies-discrete} imply that both statements in
Theorem~\ref{thm:duality-principle} are false. This is not so
for the Bessel duality. In fact,
Corollary~\ref{cor:Bessel-system-with-g-in-s0} shows that for any
closed subgroup $\Delta$ in $G \times \ghat$ \emph{any} function $g$ in
the Feichtinger algebra $S_0(G)$ will generate a Bessel system $\gaborG{g}$.
It is thus the additional lower frame inequality and lower Riesz family
condition that restrict the interesting (non-empty) statements of
Theorem~\ref{thm:duality-principle} to the
case where $(G\times\ghat)/\Delta$ is compact. It is, however,
remarkable that both of the assumptions limit the admissible subgroups
$\Delta$ to exactly those that have a compact quotient
$(G\times\ghat)/\Delta$. A similar comment holds for 
the Wexler-Raz biorthogonality relations in Theorem~\ref{th:wex-raz-nonsep}.


\begin{corollary}
\label{col:WR-and-duality-tight}
Let $\Delta$ be a closed subgroup
of $G\times\ghat$, and let $g\in L^2(G)$. Then $\gaborG{g}$ is a tight frame with bound $A$ if, and only if, $\gaborG[\Delta^{\circ}]{g}$ is an orthogonal system with $\norm{g}^2 = \vol{\Delta} A$.
\end{corollary}
\begin{proof} 
  One implication follows by Corollary~\ref{col:WR-tight}. For the other note that an orthogonal
  system $\{\pi(\mu)g\}_{\mu\in\Delta^{\circ}}$ with $\norm{g}^2 = \vol{\Delta} A$ is a basic Riesz
  family, where both bounds are $\vol{\Delta} A$. By Theorem~\ref{thm:duality-principle} the Gabor system
  $\gabor{g}$ is a tight frame with bound $A$.
\end{proof}

Now, we are ready to show existence of tight Gabor frames in
$L^2(\R^n)$ for very ``sparse'', but non-discrete,
subgroups $\Delta$; here we mean sparse (or thin) in the sense that
$\vol{\Delta}$ can be arbitrarily large. On the other hand, if
$\Delta$ is discrete, we saw in
Theorem~\ref{thm:density-frame-uniform} that $\vol{\Delta} \le 1$ is
necessary for the existence of Gabor frames. 

\begin{example}
\label{exa:construction-of-frames-in-L2Rn}
Let $G=\R^n$ and let $\Delta=\LL \times \LG$, where $\LL$ and $\LG$
are closed, co-compact subgroups of $\R^n$. Then $\LG = P(\Z^r \times
\R^{n-r})$ and $\LL = Q(\Z^s \times \R^{n-s})$ for some $P,Q \in
\mathrm{GL}_{\R}(n)$ and $0 \le r,s \le n$. If we consider $P$ and $Q$
as $n \times n$ matrices and the columns
of $P$ and $Q$ as vectors in $\R^n$, we can take the last $n-r$ and
the last $n-s$ columns of $P$ and $Q$, respectively, to be orthonormal
vectors. We then equip $\LG^\perp = (P^T)^{-1}(\Z^r \times
\{0\}^{n-r})$ and $\LL^\perp = (Q^T)^{-1}(\Z^s \times \{0\}^{n-s})$
with the counting measure times $\abs{\det{P}}^{-1}$ and
$\abs{\det{Q}}^{-1}$, respectively. It follows that $\vol{\Delta}=\abs{\det{(PQ)}}$. We split the construction of
tight Gabor frames $\sepgabor{\LL}{\LG}$ in three cases:
\begin{enumerate}[(a)]
\item $r<n$, any $P,Q \in \mathrm{GL}_{\R}(n)$,
\item $s<n$, any $P,Q \in \mathrm{GL}_{\R}(n)$,
\item $r=n$ and $s=n$, and $\vol{\Delta}=\abs{\det{(PQ)}}\le 1$ for $P,Q \in \mathrm{GL}_{\R}(n)$.
\end{enumerate}
Case (a) and (b) correspond to non-discrete subgroups, while case (c)
is the well-known setup of discrete Gabor systems.
\begin{enumerate}[(a)]
\item 
 Applying the dilation operator $D_{Q^{-1}}$, defined on $L^2(\R^n)$ by
$D_Af(x)=\det{A}^{1/2}f(Ax)$, to the functions in $\sepgabor{Q(\Z^s \times \R^{n-s})}{P(\Z^r \times \R^{n-r})}$ we obtain: 
\[ 
\sepgabor[\tilde{g}]{\Z^s \times \R^{n-s}}{QP(\Z^r \times \R^{n-r})}, \quad \text{with } \tilde{g}:=D_{Q^{-1}}g .
\]
The adjoint Gabor system is 
\begin{equation}
\set{E_\beta T_\alpha \tilde{g}}_{\alpha \in A(\Z^r \times
  \{0\}^{n-r}), \beta \in \Z^s \times \{0\}^{n-s}}, \label{eq:Rn-ex-adjoint}
\end{equation}
where $A=((QP)^{T})^{-1}$. We will choose $\tilde{g}$ so that this adjoint system is an orthonormal
system. By Corollary~\ref{col:WR-and-duality-tight}, the Gabor system $\sepgabor{\LL}{\LG}$ generated
by $g=D_Q\tilde{g}$ will then be a Parseval frame for $L^2(\R^n)$.

Obviously, the system $\set{E_\beta \mathds{1}_{\itvcc{0}{1}^n}}_{\beta \in \Z^s \times \{0\}^{n-s}}$ is
orthonormal. We consider the columns of $A$ as vectors in $\R^n$ and redefine the last $n-r$ columns of $A$ to be an
orthonormal basis of the orthogonal complement of the first $r$ column
vectors of $A$. For each $z\in \Z^r$ define
\[ 
K_z = \itvcc{0}{1}^n \cap A(\itvcc{z}{z+1}^r \times \R^{n-r}).
\]
Let $Z=\setprop{z \in \Z^r}{K_z\neq \emptyset}$; as usual, our set
relations should be understood only up to sets of measure zero. Since $Z$ is
finite and the subspace $A(\R^r \times \{0\}^{n-r})$ of co-dimension $n-r>0$,
we can find points $\setprop{y_z\in \Z^n}{z \in Z}$ that satisfy
\begin{equation}
(K_z+y_z)\cap (K_{z'}+y_{z'}+\alpha) = \emptyset \quad  \forall \alpha
\in A(\Z^r \times \{0\}^{n-r})\setminus \{0\}^n \label{eq:Kz-yz-def}
\end{equation}
for all $z,z' \in Z$. The choice of $y_z$ is
illustrated in Figure~\ref{fig:Kz}.
\begin{figure}[ht]
  \centering
    \begin{tikzpicture}[scale=3.2]
    \coordinate (Origin)   at (0,0);
    \coordinate (A)   at (0,0);
    \coordinate (B)   at (.5,0);
    \coordinate (C)   at (1,0);
    \coordinate (D)   at (0,1);
    \coordinate (E)   at (0.5,1);
    \coordinate (F)   at (1,1);
\draw[->, thick] (-1.1,0) -- (2.1,0) node[right] {$x_1$};
\draw[->, thick] (0,-1.1) -- (0,2.1) node[above] {$x_2$};
    \pgftransformcm{1}{0}{0}{1}{\pgfpoint{0cm}{0cm}}
    \coordinate (Aone) at (0.4,0.2);
    \draw[style=help lines,dashed] (-1,-1) grid[step=1cm] (2,2);
    \foreach \x in {-2,-1,...,5}{
        \node[draw,circle,inner sep=1pt,fill] at (0.4 *\x,0.2*\x) {};
    }
    \draw [thick,-latex] (Origin)
        -> (Aone); 
    \filldraw[fill=gray, fill opacity=0.1, draw=black] (Origin) -- (B)
    -- (D) -- cycle;
  \filldraw[fill=yellow, fill opacity=0.1, draw=black] (-1,1) -- ($(B)+(-1,1)$)
    -- ($(D)+(-1,1)$) -- cycle;
    \filldraw[fill=gray, fill opacity=0.3, draw=black] (B) -- (C)
    -- (E) -- (D) -- cycle; 
   \filldraw[fill=gray, fill opacity=0.5, draw=black] (C) -- (E)
    -- (F) -- cycle;
  \filldraw[fill=blue, fill opacity=0.5, draw=black] ($(C)+(1,-1)$) -- ($(E)+(1,-1)$)
    -- ($(F)+(1,-1)$) -- cycle;
\foreach \x/\xtext in {-1/-1, 1/1}
    \draw[shift={(\x,0)}] (0pt,1pt) -- (0pt,-1pt) node[below left]
    {$\xtext$};
\foreach \y/\ytext in {-1/-1,1/1, 2/2}
    \draw[shift={(0,\y)}] (1pt,0pt) -- (-1pt,0pt) node[above left] {$\ytext$};
\node at (-.77,1.08) {$K_0+y_0$};
\node at (0.45,0.6) {$K_1=$};
\node at (0.15,0.30) {$K_0$};
\node at (0.37,0.08) {$a_1$};
\node at (0.53,0.45) {$K_1+y_1$};
\node at (0.85,0.85) {$K_2$};
\node at (1.78,-.1) {$K_2+y_2$};
  \end{tikzpicture}
  \caption{An example for $n=2, r=1$ showing the choice of the integer vectors
    $y_z$ in \eqref{eq:Kz-yz-def}. The dots show $A(\Z \times \{0\})$, where
    $A=\left[a_1 \; a_2\right]$ is a $2\times 2$ matrix, and the column vector $a_1$ 
    is illustrated as a geometric vector on the plot, and $a_2$ is
    orthogonal to $a_1$. Then $Z=\{0,1,2\}$, and we can take $y_0=(-1,1),
    y_1=(0,0), y_2=(1,-1)$. With this choice the set $\cup_{z \in Z} (K_z+y_z)$,
    and its translates along the ``dots'' $A(\Z \times
    \{0\}) \setminus \{(0,0)\}$ are disjoint.}
  \label{fig:Kz}
\end{figure}
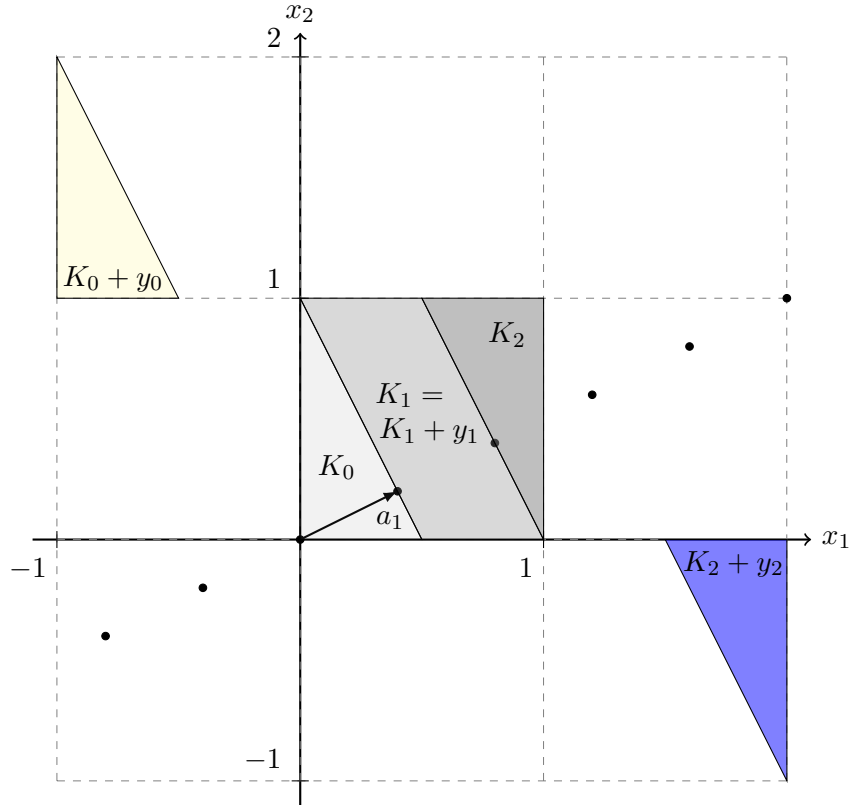
Define $\tilde{g} \in L^2(\R)$ by 
\[
\tilde{g}=\sum_{z \in Z}\mathds{1}_{K_z+y_z}.
\]
By $\Z^n$-periodicity of $\myexp{2\pi i \beta x}$ for $\beta \in \Z^n$ 
and the fact that $y_z\in \Z^n$, we see that $\set{E_\beta
  \tilde{g}}_{\beta \in \Z^s \times \{0\}^{n-s}}$ is an orthonormal
set. By \eqref{eq:Kz-yz-def}, the translates $T_\alpha\tilde{g}$ and
$T_{\alpha'}\tilde{g}$ have disjoint support for $\alpha,\alpha' \in
A(\Z^r \times \{0\}^{n-r})$ whenever $\alpha \neq \alpha'$. Combining
these two facts, we see that the adjoint Gabor
system~\eqref{eq:Rn-ex-adjoint} is orthonormal.  As a conclusion we
have that tight Gabor frames $\sepgabor{\LL}{\LG}$ exist for any value
of $\vol{\Delta}=\abs{\det{(PQ)}}>0$ with generators $g$ having
compact support.  
\item 
An application of the Fourier transform switches the
role of $\LL$ and $\LG$. Hence, we can construct a tight Gabor frame generator $g$
in the frequency domain by directly referring to case (a). This approach, however, leads to
bandlimited generators. If compactly supported generators are desired,
slight modifications of the procedure in (a) are necessary.  
\item 
When $r=n$ and $s=n$, both $\LL=P\Z^n$ and $\LG=Q\Z^n$ are full-rank lattice in
$\R^n$, and we equip these discrete subgroups with the counting measure. Under this setup, Han and
Wang~\cite{MR1866800} prove that $\vol{\Delta}=\abs{\det{(PQ)}}\le 1$ is equivalent with the existence of a tight Gabor
frames $\sepgabor{\LL}{\LG}$. Their proof is constructive and the constructed Gabor windows
$g$ is, as above, a characteristic function of a given set. However, 
in this case the set might be unbounded, in which case the generator
will not have compact support.
\end{enumerate}
\end{example}

For uniform lattices $\Delta$ in $G \times \ghat$ it follows from the duality principle in
Theorem~\ref{thm:duality-principle} applied to the density result in
Theorem~\ref{thm:density-frame-uniform} that $\gaborG{g}$ being a (basic) Riesz family implies that
$\vol{\Delta}\ge 1$. For non-discrete Gabor systems this conclusion is false, in fact, by the
duality principle, Example~\ref{exa:construction-of-frames-in-L2Rn} shows that $\vol{\Delta}$ can
take any value in $\R_+$.



\appendix

\section{The Feichtinger algebra $S_{0}$}
\label{sec:proof-lemma-xx}

For functions $f,g\in L^1(G)$ involution and convolution are defined
by 
\[ f^{\dagger}(x) = \overline{f(x^{-1})} \quad \text{and} \quad (f * g)(x) = \int_{G}
f(s) g(xs^{-1})\,ds,  
\]
respectively. The function space $S_0$ is a Fourier-invariant Banach space that is dense in $L^2$ and whose
members are continuous and integrable functions. 
It is can be defined as follows.
\begin{definition} \label{def:s0-by-modulation-and-convolution} For an LCA group $G$ we define
\[ S_{0}(G) = \Big\lbrace  f\in L^{1}(G) \, : \, \int_{\ghat} \Vert E_{\omega} f * f^{\dagger} \Vert_{L^{1}} \, d\omega < \infty \Big\rbrace .\]
We endow $S_{0}(G)$ with the norm $\Vert f \Vert_{S_{0},g}:= \Vert \mathcal{V}_{g}f\Vert_{L^{1}(G\times\ghat)}$ for a fixed $g\in S_{0}(G)$.
\end{definition}

The space $S_{0}(G)$ is a Banach algebra under convolution and
pointwise multiplication, also known as the Feichtinger algebra \cite{MR643206}. It is
a special instance of both a modulation space and a Wiener
amalgam space, namely, $M^1$ and $W(\mathcal{F}L^{1},L^{1})$. 
The following relations explain the definition
of $\Vert f \Vert_{S_{0},g}$: 
\begin{equation}
   \int_{\ghat} \Vert E_{\omega}f*g^{\dagger} \Vert_{L^{1}(G)} \, d\omega = \int_{G} \Vert \mathcal F(f\cdot \overline{T_{x}g}) \Vert_{L^{1}(\ghat)} \, dx = \Vert \mathcal{V}_{g}f
  \Vert_{L^{1}(G\times\ghat)}  \label{def:s0-by-stft-modulation-relation-eq}
\end{equation}
for $f,g \in S_0(G)$, where we have used that
\begin{equation} \label{eq:stft-rewrite-to-wfl1l1} 
\mathcal{V}_{g}f(x,\omega) = \mathcal{F}\big(f\cdot \overline{T_{x}g}\big)(\omega), \ \
(x,\omega)\in G\times\ghat, \quad \text{for all $f,g\in L^{2}(G)$.} 
\end{equation}


In the proof of Theorem~\ref{thm:Bessel-duality} we need the following two properties of $S_{0}$:
\begin{itemize}
\item The product of two short-time Fourier transforms of $L^2$-functions with windows in
  $S_0(G)$ is a function in $S_0(G \times \ghat)$ (Theorem~\ref{thm:frame-op-is-in-s0-for-g-in-s0}).
\item For any windows $g \in L^2(G)$ and any closed subgroup $\Delta$
  of $G \times \ghat$, the frame operator 
\[ 
S_{g,\Delta} : L^2(G) \to L^2(G), \quad S_{g,\Delta}f=\int_{\Delta} \vert \langle f,\pi(\nu)g\rangle \vert^2 \, d\nu
\]
with domain $D(S_{g,\Delta})=S_0(G) \subset L^2(G)$ is well-defined
(Corollary~\ref{cor:Bessel-system-with-g-in-s0}).
\end{itemize}

The aim of this appendix is to give a proof of these statements.
The material presented here is known in
the lattice case in $L^2(\R^n)$ \cite{MR2020210,MR2264211,MR1601107,MR1843717}, and the generalization to $L^2(G)$ is
routine using standard harmonic analysis. We have included the proofs for completeness. Along the
way, we obtain a direct proof of the H\"older inequalities for certain Wiener amalgam
spaces. 



We need some further properties of short-time Fourier transform.
\begin{lemma} \label{le:STFT-properties-lemma} 
Let $g,g_{i},f,f_{i}\in L^2(G)$, $i=1,2$ and $x,\alpha\in G$ and $\omega,\beta\in\ghat$. 
Then the short-time Fourier transform
\[ \mathcal{V}_{g}: L^2(G)\to L^2(G\times\ghat), \quad \mathcal{V}_{g}f(x,\omega) = \langle f, E_{\omega}T_{x}g\rangle\] 
satisfies the following relations:
\begin{enumerate}[(a)]
\item $ \mathcal{V}_{g}E_{\beta}T_{\alpha}f = \beta(\alpha) \,
  E_{(e_{\ghat},\alpha^{-1})}T_{(\alpha,\beta)} \mathcal{V}_{g}f$, where $e_{\ghat}$ denotes the
  identity element in $\ghat$, \label{item:2}
\item $\mathcal{V}_{E_{\beta}T_{\alpha}g}E_{\beta}T_{\alpha}f = \beta(x) \overline{\omega(\alpha)}
  \, \mathcal{V}_{g}f$, \label{item:3}
\item $\mathcal F ( \mathcal{V}_{g_1}f_{1} \cdot \overline{\mathcal{V}_{g_2}f_2})(\beta,\alpha) =
  \langle f_1,E_{\beta}T_{\alpha^{-1}}f_2\rangle\langle E_{\beta}T_{\alpha^{-1}}g_2,g_1\rangle$,
  where $\mathcal F$ is the Fourier transform on $G\times\ghat$. \label{item:4}
\end{enumerate}
\end{lemma}
\begin{proof} 
  Assertion~\eqref{item:2} follows from: 
\begin{multline*}
 (\mathcal{V}_{g}E_{\beta}T_{\alpha}f)(x,\omega)  = \langle E_{\beta}T_{\alpha}f, E_{\omega}T_{x}g\rangle = 
 \langle f, T_{\alpha^{-1}}E_{\omega\beta^{-1}}T_{x}g\rangle
= \overline{\omega(\alpha)} \beta(\alpha) \langle f, E_{\omega\beta^{-1}}T_{x\alpha^{-1}}g\rangle \\  = \overline{\omega(\alpha)} \beta(\alpha)  \mathcal{V}_{g}f(x\alpha^{-1},\omega\beta^{-1}) = \beta(\alpha) \, \big(E_{(e_{\ghat},\alpha^{-1})}T_{(\alpha,\beta)} \mathcal{V}_{g}f\big)(x,\omega). 
\end{multline*}
Assertion~\eqref{item:3} follows by similar manipulations, using the unitarity of $E_{\beta}T_{\alpha}$:
\begin{multline*}
 \mathcal{V}_{E_{\beta}T_{\alpha}g}E_{\beta}T_{\alpha}f(x,\omega) = \langle E_{\beta}T_{\alpha}f,E_{\omega}T_{x}E_{\beta}T_{\alpha}g\rangle = \beta(x) \overline{\omega(\alpha)} \langle E_{\beta}T_{\alpha}f,E_{\beta}T_{\alpha}E_{\omega}T_{x}g\rangle \\
 = \beta(x) \overline{\omega(\alpha)} \langle f,E_{\omega}T_{x}g\rangle = \beta(x) \overline{\omega(\alpha)} \mathcal{V}_{g}f(x,\omega).
\end{multline*}
For~\eqref{item:4} we do the following:
\begin{align*}
& \mathcal F ( \mathcal{V}_{g_1}f_{1} \cdot \overline{\mathcal{V}_{g_2}f_2} )(\beta,\alpha) = \int_{G\times\ghat} \mathcal{V}_{g_1}f_{1}(x,\omega) \cdot \overline{\mathcal{V}_{g_2}f_2(x,\omega) \beta(x) \omega(\alpha)} \, d(x,\omega) \\
& \overset{\mathclap{\eqref{item:3}}}=\int_{G\times\ghat} \mathcal{V}_{g_1}f_{1}(x,\omega) \cdot \overline{\mathcal{V}_{E_{\beta}T_{\alpha^{-1}}g_2}E_{\beta}T_{\alpha^{-1}}f_2(x,\omega)} \, d(x,\omega)
= \langle \mathcal{V}_{g_1}f_{1}, \mathcal{V}_{E_{\beta}T_{\alpha^{-1}}g_2}E_{\beta}T_{\alpha^{-1}}f_2\rangle \\
& \overset{\mathclap{\eqref{eq:STFT-dual}}}= \langle f_1,E_{\beta}T_{\alpha^{-1}}f_2\rangle\langle E_{\beta}T_{\alpha^{-1}}g_{2},g_{1}\rangle.
\end{align*}
\end{proof}

The norm on $S_0(G)$ depends on a fixed
function $g\in S_0(G)$. However, any function $g$ in $S_0$ induces an equivalent norm. Indeed, for
$f,g_1,g_2\in S_0(G)$ one can show that
\[ 
\Vert g_1 \Vert_{L^2}^2 \, \Vert g_2 \Vert_{S_0,g_1}^{-1} \, \Vert f
\Vert_{S_0,g_2} \le \Vert f \Vert_{S_0,g_1} \le \Vert g_2
\Vert_{L^2}^{-2} \, \Vert g_2 \Vert_{S_0,g_1} \, \Vert f
\Vert_{S_0,g_2} .
\]
As a consequence, a function $f\in L^{1}(G)$ belongs to $S_0(G)$
if, and only if, $\mathcal{V}_{g}f\in L^1(G\times\ghat)$ for any and thus all $g\in S_0(G)$. 

\begin{proposition} \label{prop:Vfg-in-s0}
If $f,g \in S_0(G)$, then $\mathcal{V}_{g}f\in
S_{0}(G\times\ghat)$ and 
\[ 
\Vert \mathcal{V}_{g}f \Vert_{S_0,\mathcal{V}_{g}f} = \Vert f
\Vert_{S_{0},f} \, \Vert g \Vert_{S_{0},g}.  
\]
\end{proposition}
\begin{proof} Let $f,g\in S_{0}(G)$. By the argument preceding the proposition, we have that $\mathcal{V}_{g}f\in L^{1}(G\times\ghat)$.
Using Lemma~\ref{le:STFT-properties-lemma} we find: 
\begin{align*}
& \int_{G\times\ghat}\int_{\ghat\times G} \vert \langle \mathcal{V}_{g}f, E_{(\beta,\alpha)}T_{(x,\omega)} \mathcal{V}_{g}f\rangle \vert \, d(\beta,\alpha) \, d(x,\omega) \\
& \overset{\mathclap{\eqref{item:2}}} = \int_{G\times\ghat}\int_{\ghat\times G} \vert \langle \mathcal{V}_{g}f, E_{(\beta,\alpha x)} \mathcal{V}_{g}E_{\omega}T_{x}f\rangle \vert \, d(\beta,\alpha) \, d(x,\omega)\\
& \overset{\mathclap{\strut\text{(\ref{eq:stft-rewrite-to-wfl1l1})}}} = \int_{G\times\ghat}\int_{\ghat\times G} \big\vert \mathcal{F}\big( \mathcal{V}_{g}f \cdot \overline{\mathcal{V}_{g}E_{\omega}T_{x}f} \big)(\beta,\alpha x) \big\vert \, d(\beta,\alpha) \, d(x,\omega) \\
& \overset{\mathclap{\eqref{item:4}}} = \int_{G\times\ghat}\int_{\ghat\times G} \big\vert \langle f,E_{\beta\omega}T_{\alpha^{-1}}f\rangle\langle E_{\beta}T_{x^{-1}\alpha^{-1}}g,g\rangle  \big\vert \, d(\beta,\alpha) \, d(x,\omega) \\
 \big(\alpha\mapsto\alpha^{-1},\beta\mapsto\beta\omega^{-1} \big)&  = \int_{G\times\ghat}\int_{\ghat\times G} \big\vert \langle f,E_{\beta}T_{\alpha}f\rangle\langle E_{\beta\omega^{-1}}T_{\alpha x^{-1}}g,g\rangle  \big\vert \, d(\beta,\alpha) \, d(x,\omega) \\
& = \int_{G\times\ghat} \vert \mathcal{V}_{f}f(\alpha,\beta) \vert \int_{G\times\ghat} \vert \mathcal{V}_{g}g(\alpha x^{-1},\beta \omega^{-1}) \vert \, d(x,\omega) \, d(\alpha,\beta) \\
& = \Vert f \Vert_{S_{0},f} \, \Vert g \Vert_{S_{0},g}.
\end{align*}
\end{proof}


\begin{theorem} \label{thm:frame-op-is-in-s0-for-g-in-s0} Let $f_{i}\in L^2(G)$ and $g_{i}\in S_0(G)$, $i=1,2$.
Then the mapping 
\[\varphi : G\times\ghat \to \C, \ (x,\omega)\mapsto \big(\mathcal{V}_{g_{1}}f_{1}\cdot \overline{\mathcal{V}_{g_{2}}f_{2}}\big)(x,\omega)\] belongs to $S_0(G\times\ghat)$.
\end{theorem}
\begin{proof} It is clear that $\varphi\in L^{1}(G\times\ghat)$. 
Now, let $g_{0}\in S_{0}(G)$ and define
$\varphi_{0}:=\mathcal{V}_{g_{0}}g_{0}$. By
Proposition~\ref{prop:Vfg-in-s0} the function $\varphi_{0}\in
S_{0}(G\times\ghat)$, and thus $\varphi_{0}^2\in S_{0}(G\times\ghat)$. 
To finish the proof, it suffices to show that $\Vert \varphi \Vert_{S_0,\varphi_{0}^{2}} = \Vert \mathcal{V}_{\varphi_{0}^2}
\varphi \Vert_{L^{1}(G\times\ghat\times \ghat\times G)} <\infty$. We
show this in two
steps. 

 Step 1: Using the definition of the short-time Fourier transform and \eqref{eq:stft-rewrite-to-wfl1l1} we find that
\begin{align}
\nonumber & \Vert \mathcal{V}_{\varphi_{0}^2} \varphi \Vert_{L^{1}} = \int_{G\times\ghat} \int_{\ghat\times G} \vert \langle \varphi, E_{(\beta,\alpha)}T_{(x,\omega)} \varphi_{0}^2\rangle \vert \, d(\beta,\alpha) \, d(x,\omega) \\
\nonumber & = \int_{G\times\ghat} \int_{\ghat\times G} \vert \mathcal{F}\big( \mathcal{V}_{g_{1}}f_{1}\cdot \overline{\mathcal{V}_{g_{2}}f_{2}}\cdot\overline{T_{(x,\omega)} \varphi_{0}^2} \big) (\beta,\alpha) \vert \, d(\beta,\alpha) \, d(x,\omega) \\
\nonumber & = \int_{G\times\ghat} \Vert \mathcal{F} (\mathcal{V}_{g_{1}}f_{1} \, \overline{T_{(x,\omega)} \varphi_{0}}) * \mathcal{F}(\overline{\mathcal{V}_{g_{2}}f_{2} \, T_{(x,\omega)} \varphi_{0}}) \Vert_{L^{1}(\ghat\times G)} \, d(x,\omega) \\
\nonumber & \le \int_{G\times\ghat} \Vert \mathcal{F} (\mathcal{V}_{g_{1}}f_{1} \, \overline{T_{(x,\omega)} \varphi_{0}}) \Vert_{L^{1}(\ghat\times G)} \, \Vert \mathcal{F}(\overline{\mathcal{V}_{g_{2}}f_{2} \, T_{(x,\omega)} \varphi_{0}}) \Vert_{L^{1}(\ghat\times G)} \, d(x,\omega) \\
\label{eq:proof-frame-op-is-in-s0-for-g-in-s0-step1} & \le 
\Big( \int_{G\times\ghat} \Vert \mathcal{F} (\mathcal{V}_{g_{1}}f_{1} \, \overline{T_{(x,\omega)} \varphi_{0}}) \Vert_{L^{1}}^{2} \, d(x,\omega) \Big)^{1/2} \, \Big( \int_{G\times\ghat} \Vert \mathcal{F}(\overline{\mathcal{V}_{g_{2}}f_{2} \, T_{(x,\omega)} \varphi_{0}}) \Vert_{L^{1}}^{2} \, d(x,\omega) \Big)^{1/2}
\end{align}

Step 2: We now show that both factors in \eqref{eq:proof-frame-op-is-in-s0-for-g-in-s0-step1} are finite. Consider the first of the factors. By use of Lemma~\ref{le:STFT-properties-lemma} find the following.
\begin{align}
\nonumber & \int_{G\times\ghat} \Vert \mathcal{F} (\mathcal{V}_{g_{1}}f_{1} \, \overline{T_{(x,\omega)} \mathcal{V}_{g_{0}}g_{0}}) \Vert_{L^{1}}^{2} \, d(x,\omega) \\
\nonumber & \overset{\mathclap{\eqref{item:2}}} = \int_{G\times\ghat} \Big( \int_{\ghat\times G} \vert \mathcal{F}\big( \mathcal{V}_{g_{1}}f_{1} \, \overline{\mathcal{V}_{g_{0}} E_{\omega}T_{x}g_{0}}\big)(\beta,\alpha x)\vert \, d(\beta,\alpha) \Big)^2 \, d(x,\omega) 
\\
\label{eq:proof-frame-op-is-in-s0-for-g-in-s0-step2}  & \overset{\mathclap{\eqref{item:4}}}  = \int_{G\times\ghat} \Big( \int_{\ghat\times G} \vert \langle f_{1}, E_{\beta}T_{\alpha^{-1}x^{-1}}E_{\omega}T_{x}g_{0}\rangle \langle E_{\beta}T_{\alpha^{-1}x^{-1}}g_{0},g_{1}\rangle \vert \, d(\beta,\alpha) \Big)^2 \, d(x,\omega)\end{align}
By use of the short-time Fourier transform, expansion of the square term and a change of variables $\alpha\mapsto \alpha^{-1}x^{-1}$ we rewrite \eqref{eq:proof-frame-op-is-in-s0-for-g-in-s0-step2} to yield the following:
\begin{align*}
& \int_{G\times\ghat} \Vert \mathcal{F} (\mathcal{V}_{g_{1}}f_{1} \, \overline{T_{(x,\omega)} \mathcal{V}_{g_{0}}g_{0}}) \Vert_{L^{1}}^{2} \, d(x,\omega) \\
& = \iiint \vert \mathcal{V}_{g_{0}}f_{1}(x\alpha,\omega\beta)\vert \, \vert \mathcal{V}_{g_{0}}f_{1}(x\tilde\alpha,\omega\tilde\beta)\vert\,  \vert \mathcal{V}_{g_{0}}g_{1}(\alpha,\beta)\vert \,   \vert \mathcal{V}_{g_{0}}g_{1}(\tilde\alpha,\tilde\beta)\vert \, d(\alpha,\beta) \, d(\tilde\alpha,\tilde\beta) \, d(x,\omega) \\
& = \int \vert \mathcal{V}_{g_{0}}g_{1}(\alpha,\beta)\vert \int \vert \mathcal{V}_{g_{0}}g_{1}(\tilde\alpha,\tilde\beta)\vert \int \vert \mathcal{V}_{g_{0}}f_{1}(x\alpha,\omega\beta)\vert \, \vert \mathcal{V}_{g_{0}}f_{1}(x\tilde\alpha,\omega\tilde\beta)\vert \, d(x,\omega) \, d(\tilde\alpha,\tilde\beta) \, d(\alpha,\beta)\\
& \le \Vert \mathcal{V}_{g_{0}}g_{1} \Vert_{L^{1}}^2 \, \Vert \mathcal{V}_{g_{0}}f_{1} \Vert_{L^{2}}^2 = \Vert g_{1} \Vert_{S_{0},g_{0}}^2 \Vert g_{0} \Vert_{L^{2}}^{2} \, \Vert f_{1} \Vert_{L^{2}}^{2},
\end{align*}
where all integrals are over ${G \times \ghat}$.
The bound for the other term in \eqref{eq:proof-frame-op-is-in-s0-for-g-in-s0-step1} is obtained similarly. Combining step 1 and 2 yields that
\[ 
\Vert \varphi \Vert_{S_{0},\varphi_{0}^{2}} = \Vert \mathcal{V}_{\varphi_{0}^2} \varphi
\Vert_{L^{1}} \le \Vert g_{0} \Vert_{L^{2}}^2 \, \Vert g_{1} \Vert_{S_{0},g_{0}} \, \Vert g_{2}
\Vert_{S_{0},g_{0}} \, \Vert f_{1} \Vert_{L^2} \, \Vert f_{2} \Vert_{L^2} , 
\]
where $\varphi_{0}=\mathcal{V}_{g_{0}}g_{0}$, $g_{0}\in S_{0}(G)$.
\end{proof}

Step 1 in the proof of Theorem~\ref{thm:frame-op-is-in-s0-for-g-in-s0} shows that
\[ 
\Vert f\cdot g \Vert_{W(\mathcal{F}L^{1},L^{1})} \le \Vert f \Vert_{W(\mathcal{F}L^{1},L^{2})} \,
\Vert g \Vert_{W(\mathcal{F}L^{1},L^{2})}.
\]
Using the H\"older inequality rather than the Cauchy-Schwarz inequality in
\eqref{eq:proof-frame-op-is-in-s0-for-g-in-s0-step1} yields 
a H\"older inequality for Wiener amalgam spaces:
\begin{equation}
\label{eq:mod-spaces-holder-ineq} 
\Vert f\cdot g \Vert_{W(\mathcal{F}L^{1},L^{1})}
\le \Vert f \Vert_{W(\mathcal{F}L^{1},L^{p})} \, \Vert g \Vert_{W(\mathcal{F}L^{1},L^{q})}, \ \ 1
= 1/p+1/q, \ 1\le p,q\le \infty
\end{equation}
for $f \in W(\mathcal{F}L^{1},L^{p})$ and $g \in W(\mathcal{F}L^{1},L^{q})$.
In the special case of $G=\R^n$ the inequality \eqref{eq:mod-spaces-holder-ineq} plays an important
role in \cite{MR2020210,MR2264211}. On the other hand, \eqref{eq:mod-spaces-holder-ineq} holds for
more general Wiener amalgam spaces \cite{MR751019}.

From \cite[Theorem 7]{MR643206} we have the following important property of $S_0$. For any closed
subgroup $H$ of $G$ the restriction mapping
\[ R_{H}:S_{0}(G)\to S_{0}(H), \quad R_{H}f(x)=f(x), \ x\in H\]
is onto and bounded.

\begin{corollary} 
\label{cor:Bessel-system-with-g-in-s0} 
Let $\Delta$ be a closed subgroup of $G\times\ghat$.  If $g\in S_0(G)$, then there exists a constant
$K>0$ (which only depends on $\Delta$) such that
\[ 
\int_{\Delta} \vert \langle f,\pi(\nu)g\rangle \vert^2 \, d\nu \le K \, \Vert g \Vert_{S_{0},g}^2 \,
\Vert g \Vert_{L^{2}}^{2} \,  \Vert f \Vert_{L^{2}}^{2} \quad  \text{for all } f\in L^2(G).
\]
\end{corollary}
\begin{proof} 
From Theorem~\ref{thm:frame-op-is-in-s0-for-g-in-s0} we have that the mapping 
\[\varphi : G\times\ghat \to \C, \ (x,\omega)\mapsto \vert \langle f, E_{\omega}T_{x} g\rangle\vert^2 \] belongs to $S_0(G\times\ghat)$ for any $f\in L^{2}(G)$.
By the comment preceding Corollary \ref{cor:Bessel-system-with-g-in-s0}, we have that for a closed subgroup $\Delta$ of $G\times\ghat$, the mapping $ \nu \mapsto \vert\langle f, \pi(\nu) g\rangle \vert^2$
also belongs to $S_0(\Delta)$. Hence, it belongs, in particular, to $L^{1}(\Delta)$. Therefore 
\[ \int_{\Delta} \vert \langle f, \pi(\nu) g\rangle \vert^2 \, d\nu = \Bigl\Vert \vert \langle f, \pi( \, \cdot \, ) g\rangle\vert^2 \Bigr\Vert_{L^{1}(\Delta)} \le C \, \Vert R_{\Delta} \varphi \Vert_{S_0} \le C \, \Vert R_{\Delta} \Vert_{op} \, \Vert \varphi \Vert_{S_0} .\]
The result now follows by the proof of Theorem~\ref{thm:frame-op-is-in-s0-for-g-in-s0}.
\end{proof}

Corollary~\ref{cor:Bessel-system-with-g-in-s0} shows that the frame operator $S_{g}$ with $g\in
S_{0}(G)$ is well-defined and bounded on $L^2(G)$. However, it also shows, and this is what we used
in Section~\ref{thm:Bessel-duality}, that the operator $S_{g}$,
$g\in L^2(G)$, is well-defined when the domain is restricted to the subspace $S_{0}(G)$ of $L^2(G)$.
We refer to~\cite{MR1601107} for further results of this nature for $G=\R^n$.



\def\cprime{$'$} \def\cprime{$'$} \def\cprime{$'$}
  \def\uarc#1{\ifmmode{\lineskiplimit=0pt\oalign{$#1$\crcr
  \hidewidth\setbox0=\hbox{\lower1ex\hbox{{\rm\char"15}}}\dp0=0pt
  \box0\hidewidth}}\else{\lineskiplimit=0pt\oalign{#1\crcr
  \hidewidth\setbox0=\hbox{\lower1ex\hbox{{\rm\char"15}}}\dp0=0pt
  \box0\hidewidth}}\relax\fi} \def\cprime{$'$} \def\cprime{$'$}
  \def\cprime{$'$} \def\cprime{$'$} \def\cprime{$'$} \def\cprime{$'$}


\end{document}